\documentclass[12pt,letter]{article}  
\usepackage[T2A]{fontenc}
\usepackage[cp866]{inputenc}
\usepackage{amsmath,amssymb,latexsym}
\usepackage{amsfonts,amssymb,mathrsfs,amscd,comment}
\usepackage{amsthm} 
\usepackage[russian,english]{babel}
\newtheorem{thm}{Theorem}

\newtheorem{lemma}{Lemma}
\theoremstyle{definition}
\newtheorem{definition}{Definition}

\newcommand{\field}[1]{\mathbb{#1}}

\newcommand{\R}{\field{R}}
\newcommand{\Z}{\field{Z}}
\newcommand{\N}{\field{N}}

\renewcommand{\P}{\field{P}}

\def\LL{\mathscr L}

\def\GRS{\operatorname{GRS}}
\def\myint{\operatorname{int}}
\def\mcap{\operatorname{cap}}
\def\CC{\mathbb C}
\def\PP{\mathbb P}

\newcommand{\ol}{\overline}
\newcommand{\wt}{\widehat}

\renewcommand{\Re}{\operatorname{Re}}
\newcommand{\mc}{\mathcal}

\newcommand{\mbb}{\mathbb}

\newcommand{\cop}{\operatorname{cap}}
\newcommand{\const}{\operatorname{const}}

\renewcommand{\Im}{\operatorname{Im}}

\newcommand{\supp}{\operatorname{supp}}
\renewcommand{\(}{\left( }
\renewcommand{\)}{\right) }
\renewcommand{\[}{\left[ }
\renewcommand{\]}{\right] }
\usepackage{color}

\begin{document}

\centerline{\Large \bf Zero distribution for Angelesco Hermite--Pad\'e polynomials.}
\vspace{1cm}

\centerline{E.~A.~Rakhmanov}

\begin{abstract}
We consider the problem of zero distribution of the first kind Hermite--Pad\'e
polynomials  associated with a vector function $\vec f = (f_1, \dots, f_s)$
whose components $f_k$ are functions with
a finite number of branch points in plane. We assume that branch sets of
component functions  are well enough separated (which constitute the
Angelesco case). Under this condition we prove a theorem on limit zero
distribution for such polynomials. The limit measures are defined in terms
of a known vector equilibrium problem.

Proof of the theorem is based on the methods developed by H.~Stahl
\cite{Sta85a}--\cite{Sta86b}, A.~A.~Gonchar and the author \cite{GoRa87},
\cite{Rak12}. These
methods obtained some further generalization in the paper in application to
systems of polynomials defined by systems of complex orthogonality
relations.

Together with the characterization of the limit zero distributions of
Hermite--Pad\'e polynomials by a vector equilibrium problem  we consider an
alternative characterization using a Riemann surface $\mc R(\vec f)$
associated with $\vec f$.  In this terms we present a more general (without
Angelesco condition) conjecture on the zero distribution of Hermite--Pad\'e
polynomials.

Bibliography: 72 items.
\end{abstract}


\section{ Introduction. }

\subsection{Statement of the main theorems.}

Let $s\in \N$ and $\vec f =(f_1, f_2, \dots, f_s )$ be a vector of analytic functions defined by their Laurent expansion at infinity
\begin{equation}
f_k(z)= \sum_{m=0}^\infty \frac{f_{m,k}}{z^m}, \quad k=0,1,\dots.
\label{1}
\end{equation}
For a natural $n\in\N$ the $n$-th vector of first kind {\it Hermit-Pad\'e
polynomials} $q_{n,0}$, $q_{n,1}$, $\dots$, $q_{n,s}$ is defined by the
following relation
\begin{equation}
R_n(z):=\bigl(
q_{n,0}+q_{n,1}f_1+q_{n,2}f_2+\dots+q_{n,s}f_s
\bigr)(z)
=O\(\frac1{z^{ns+s}}\),
\label{2}\end{equation}
as $ z\to\infty$ and the condition $q_{n,k }\in \P_n, \  k=0,1,\dots, s$
where $\P_n$  that is notation for set of all polynomials of degree at most
${n}$. Function $R_n$ defined by first equality above is called the
remainder.

The construction \eqref{1} - \eqref{2} is classical. It was introduced by Hermite in 1873 for the case $f_k(z) = \exp(k/z)$ who used it to prove that
number $e$ is transcendental. Hermite student Pad\'e investigated in details case $s = 1$ which was later called after him. For the further references we refer to \cite{NiSo88}, \cite{Nu84},
\cite{BaGr96}.



The main result of the paper is related to functions $f_k$ with finite number of branch points which we denote $\mathcal A$.
More exactly, for a fixed set $e = \{a_1,\dots,a_p\}$  of $p\geq 2$ distinct points we denote by $\mathcal{A}_e$ class of function elements 
at infinity which admit analytic continuation along any curve in the
domain $\Omega = \overline {\CC} \setminus e$  which begins at infinity.
We will assume that a function $f \in \mathcal{A}_e$ is not single
valued in $\Omega;$ some of the points in $e$ may be regular or single
valued isolated singular points, but there are at least two branch
points.
Then, we define
$$
  \mathcal A = \cup \mathcal{A}_e
$$
\noindent where union is taken over all finite sets $e$ of $p \geq 2$ distinct points in plane. In other words,  $\mathcal A$ is the space of all analytic functions in plane with a finite number of branch points. In particular, $f \in \mathcal A$ means that there is a finite set $e = e(f) \in \CC$ such that $f \in \mathcal A_e$. We will write $\vec f =(f_1, \dots,f_s) \in \mc A$ if $f_j \in \mc A$ for all component functions $f_j$. \\


For certain subclass of vector functions $\vec f \in \mc A$ we will prove weak convergence of
zero counting measures
\begin{equation}\label{Meas}
\mu_{n,k} = \frac 1n \sum_{q_{n,k}(x) = 0}\delta (x) \to \lambda_k, \quad \text{as}\quad n\to\infty
\end{equation}
 of the Hermite--Pad\'e polynomials $q_{n,k}$ in \eqref{2} as $n\to \infty$ and characterize their limits $\lambda_{k}$. Convergence $\to$ in application to measures will always mean
weak convergence.

The precise condition on $\vec f$ in our convergence theorem, which we will call the Angelesco condition, is  formulated in terms of a vector equilibrium problem associated with $\vec f$. We need to introduce some definitions first.\\

For a function $f \in \mathcal A$ we denote by $ \mc{F}_0(f)$  the set of cuts that make $f$ single valued. Formally, $ \mc{F}_0(f)$ is set of compact sets  $F \subset \CC$ satisfying the condition $f \in H(\CC \setminus F)$. By $H(\Omega)$ we denote, as usual, space of holomorphic (analytic and single valued) functions in a domain $\Omega$.

It is convenient to work with a more restricted class of cuts. We will
use the subclass $\mc{F}(f)\subset \mc{F}_0(f)$ such that any compact
$F \in \mc{F}(f)$ has a finite number of connected components not
dividing the plane and each connected component contains at least two
branch points of $f$.

For a vector function $\vec f = (f_1, \dots, f_s)$ whose components $f_k$ are from $\mc A$  we denote by $\vec {\mc F} (\vec f)$ class of all vector compact sets $\vec F = (F_1, \dots , F_s) $ such that $f_k \in \mc{F}(f_k)$ for $ k =1, \dots, s$.

By $\mc{M}(F)$ we denote set of all unit positive Borel measures on a compact set $F$.
For a fixed vector-compact set $\vec{F}=\(F_1,\dotsc,F_p\)\in\mc{F}$
we define a family of vector-measures
$$
\vec{\mc{M}}=\vec{\mc{M}}(\vec{F})=\left\{\vec{\mu}
=\(\mu_1,\dotsc,\mu_p\):\mu_j\in\mc{M}\(F_j\)\right\}.
$$

The key component of any vector equilibrium problem is the interaction matrix. The matrix $A$ associated with the Angelesco equilibrium problem is the following
\begin{equation} \label{4-n}
A=\left\|a_{ij}\right\|^s_{i,j} =1, \qquad \text{where} \quad a_{ii} = 2, \quad  \text{and} \quad
a_{ij} = 1, \ \text{for}\ i \ne j,  \quad i, j = 1, \dots , s.
\end{equation}

Accordingly, the energy of a vector measure $\vec{\mu} \in \vec{\mc M}$ associated with matrix $A$ is defined by
\begin{equation}\label{5}
\mc{E}\(\vec{\mu}\) \
= [A\vec\mu, \vec\mu] = \sum^P_{i,j=1}a_{ij}\[\mu_i,\mu_j\];
\quad
\end{equation}
where
$$
[\mu,\nu]=\int U^\nu d\mu
$$
-- mutual energy of $\mu$ and $\nu$ and $U^\nu(z) = -\int \log|z-x| d\nu(x)$ is the logarithmic potential of the measure $\nu$.

The next two lemmas assert existence of solutions of two basic equilibrium
problems associated with Angelesco case. The first lemma is well known
\cite{Beck13},~\cite{HaKu12} (see also original papers \cite{GoRa81},
~\cite{GoRa85},~\cite{LoRa88},~\cite{GoRaSo97},
\cite{Gon03},~\cite{Lap15},~\cite{MaOrRa11}).

\begin{lemma}
\label{lem1}
For a fixed $\vec F \in \vec{\mc{F}}(\vec f)$ there exist a unique
vector measure $\vec{\lambda} = \vec{\lambda}_{\vec F}\in\vec{\mc{M}} (\vec{F})$
minimizing the energy
\begin{equation}\label{6}
  \mc{E}(\vec{\lambda})=\min_{\vec{\mu}\in\vec{\mc{M}}(\vec F)}\mc{E}(\vec{\mu}).
\end{equation}
\end{lemma}
\noindent (vector-equilibrium measure for $\vec{F}$).

We define the equilibrium energy functional  $\mc{E}[\vec{F}]$ on the class $\vec {\mc F}$
\begin{equation}\label{7}
\mc{E}[\vec{F}]\ = \ \inf_{\vec{\mu}\in\vec{\mc{M}} (\vec{F})}\mc{E} (\vec \mu) =
\mc{E}(\lambda_{\vec F}),
\qquad \vec{F} \in \vec{\mc F}
 \end{equation}
and assert existence of a maximizing vector compact set   $\vec \Gamma \in \vec{\mc{F}}$
\begin{lemma}
\label{lem2}
For a fixed $\vec f \in \mc A$ there exist vector-compact set $\vec{S} \in \vec{\mc{F}} =
 \vec{\mc{F}}(\vec f)$ such that
\begin{equation} \label{8}
\mc{E}[\vec{\Gamma}]=\max_{\vec{F}\in \vec{\mc{F}}}\mc{E}[\vec{F}].
\end{equation}
\end{lemma}
The proof of lemma \ref{lem2} is presented in section 4.1 below; it follows essentially from results and methods developed in \cite{Rak12}.

It is well known that in general the extremal vector compact set $\vec{\Gamma}$
is not unique. More exactly, not unique is the part $\vec{\Gamma}$ which does
not belong to the support of associated equilibrium measure. This part may vary
in a domain determined by the equilibrium potential without changing equilibrium energy.

For the part of $\vec{\Gamma}$ which carry the equilibrium measure we introduce notation $\vec\Gamma^1 = (\Gamma^1_1, \dots, \Gamma^1_s)$.  Thus, components of $\vec\Gamma^1$ are defined by  $ \Gamma^1_j = \supp \lambda_j$. This part of $\Gamma$ is unique but uniqueness is not used in proofs.

The sets $ \Gamma_j\setminus \Gamma^1_j$ are less important. They do
not carry essential amounts of zeros of Hermite--Pad\'e polynomials and
they do not play essential role in proofs. However, they also require
some attention. They have to be selected in a certain way (last
inequality in \eqref{13} has to be satisfied). They also may contain
singularities of components of $\vec f$ (when we modify contours of
integrations we have to avoid singularities of $\vec f$). Finally, they
are involved in the definition of Angelesco condition on $\vec \Gamma$
which we introduce next.

In short, an important condition in our theorems on the zero
distribution of Hermite--Pad\'e polynomials is that components of the
extremal compact set $\Gamma(\vec f)$ are disjoint. More exactly,  we
introduce the following

\begin{definition} We say that vector compact set $\vec \Gamma$ satisfies the (strict) Angelesco condition if its components are disjoint
 \begin{equation} \label{9}
 \Gamma_i \cap  \Gamma_j = \varnothing, \qquad i \ne j. 
\end{equation}
for some choice of the sets $ \Gamma_j\setminus \Gamma^1_j$.  Consequently, the vector function $\vec f = (f_1, \dots, f_s) \in \vec { \mc A}$ is called an Angelesco vector function (system) if associated vector compact set $\vec \Gamma(\vec f)$ has the  Angelesco property  \eqref{9}.
\end{definition}


Now, we state the main result of the paper.

\begin{thm} \label{thm1}
 Let vector function $\vec f = (f_1, \dots, f_s) \in \vec {\mc A}$ and $f_k \in \mc A$ for $k =1, \dots, s$
 satisfies the Angelesco condition \eqref{9}.
Then sequences of zero counting measure $\mu_{n,k}$ in \eqref{Meas} are
weakly converges as $n\to \infty$
\begin{equation} \label{10}
\mu_{n, k} \overset  {*}  {\to}\lambda_k, \quad k = 1,2, \dots, s
\end{equation}
to the components $\lambda_k$ of the vector equilibrium measure $\vec \lambda = \vec \lambda_{\vec {\Gamma}}$
of the extremal vector compact set $\vec \Gamma$ in class $ \vec{\mc{F}}(\vec f)$ defined by  \eqref{6}.
\end{thm}

It is important to add that the therem 1 is presented not in the most general form.

We choose to define the vector compact $\vec \Gamma \in \vec {\mc F}(\vec f)$ (which is the key component of the theorem) as the solution of the ``max-min'' energy problem \eqref{8} with strict Angelesco condition \eqref{9}. With this definition it is possible to prove the existence theorem
by comparetively simple reduction to known results.

Actually,  theorem 1 is valid under the assumption that there exists vector compact $\vec \Gamma \in \vec {\mc F}(\vec f)$ with $S$-property (see definition in \eqref{12} - \eqref{14} below) and condition \eqref{9} (which may be relaxed: $ \Gamma_i \cap  \Gamma_j $ is a finite set for $i \ne j$).

Numerical experiment conducted by S. Suetin shows that the existence of such vector compact
 is a more general condition then the one stated in terms of \eqref{8}. More exactly, the class of vector function $\vec f$ such that a vector compact $\vec \Gamma \in \vec{ \mc F}(\vec f)$ with $S$-property exist is essentially larger the class in theorem \ref{thm1}.
It is more  difficult to prove existence theorem for such $S$-compacts. For a moments  the ``max-min'' energy problem \eqref{8} seems to be the only known method, but some of them do not satisfy this property.

\subsection{ Remarks. Generalizations.}
\subsubsection{Asymptotics of leading coefficient.}
An important fact about first kind Hermite--Pad\'e polynomials  is that
their leading coefficients are essentially involved in the definition
of polynomials \eqref{2}. This relation determines leading coefficients
together with other coefficients of polynomials. It turns out that
asymptotics of their magnitudes may be determined in terms of the
vector equilibrium problem \eqref{6} - \eqref{8}.

Conditions \eqref{2} determine the function $R_n$ up to a
multiplicative normalizing constant. It follows that the leading
coefficient $c_{n,k}$ of one of the polynomials $q_{n,k}(x)
=c_{n,k}Q_{n,k}(x)$ may be selected arbitrarily. Together with theorem
\ref{thm1} we have the following.

\begin{thm} \label{thm2}
If normalizing constant in \eqref{2} is selected so that the sequence
$|c_{n,k}|^{1/n}$ is convergent for one particular fixed $k$ as $n\to
\infty$ then it is convergent for any $k$ and there is a constant $c>0$
depending on the normalization such that
\begin{equation*} 
 \lim_{n\to \infty}|c_{n,k}|^{1/n} = c  e^{w_k}\ , \qquad    k = 1,2, \dots, s
\end{equation*}
where $w_k$ is the $k$-th equilibrium constant associated with extremal compact set $\vec\Gamma ({\vec f});$ see \eqref{13} below.
\end{thm}
Proofs of theorems \ref{thm1} and \ref{thm2} are connected and will be
presented simultaneously.

 \subsubsection{Asymptotics of the remainder.}
An important complement to the theorems \ref{thm1} and \ref{thm2} is the $n$-th root asymptotics for the remainder $R_n$.
\begin{thm} \label{thm3}
With a properly selected  normalizing constant in \eqref{2} we have
\begin{equation*} 
\frac {1}{n} \log |R_n(z) | \ \overset{\cop}{\to} \  - U^\lambda (z), \quad z \in \ol\CC \setminus \Gamma
\end{equation*}
where $\lambda = \sum_{k=1}^s \lambda_k$ and $\Gamma = \cup_{k=1}^s \Gamma_k$ (convergence in capacity on compact sets in $ \ol\CC \setminus \Gamma$).
\end{thm}
There are indications coming from semiclassical classes of $\vec f$
that this theorem may be substantially generalized in the following
direction. Function $ - U^\lambda(z)$ which is harmonic in $\Omega =
\ol\CC \setminus \Gamma$ has a harmonic extension $g(z)$ from the domain
$\Omega $ in plane to $s+1$-sheeted algebraic Riemann surface $\mc R =
\mc R(\vec f)$.

We will go into some details related to  $\mc R $ in section 4; here we mention that on the other $s$ domains $\Omega^{(k)} \subset \mc R$ whose projection to the plane is $\Omega$ the extended function is $g(z) = U^{\lambda_k}(z) + C$ where $C$ is a common constant. Function $R(z)$ has multivalued analytic continuation to $\mc R$. We conjecture that for any domain $\mc D \subset \mc R$
where holomorphic branch of $R_n$ exist (for any $n$) we have
\begin{equation*} 
\frac {1}{n} \log |R_n(z) | \ \overset{\cop}{\to} \  g (z), \quad z \in \mc D
\end{equation*}
The conjecture is supported by the results from \cite{MaRaSu16} related
to function $\vec f$ from the class $f_k \in \mathscr L$ of the form $f_k(z)
= \prod_{i = 1}^{m_k} (z - a_{i,k})^{ \alpha_{i,k}}$ with $\sum_{i =
1}^{m_k}\alpha_{i,k} = 0$ (see also~\cite{KoSu15} and \cite{Sue15}). For function of this class the weighted
Hermite--Pad\'e polynomials $q_{n,k}(z)f_k(z)$ and the remainder
satisfy the same differential equation with polynomial coefficient of
order $s+1$ (polynomials - coefficient depend on $n$ but their degrees
are bounded by constants not depending on $n$). Combining the this fact
with the results of our paper we can prove the conjecture for the case
when $\vec f \in \mathscr L$ and the Angelesco condition is satisfied.

\subsubsection{Angelesco condition.}
Angelesco condition \eqref{8} - \eqref{9} is crucial for theorem \ref{thm1};
if it is essentially violated (say, two components of $\vec \Gamma^1$ have a
common arc) then
then the zero distribution of Hermite--Pad\'e  polynomilas is actually determined by
a different $S$-compact. At the same time, in such a case we do not have general tools
for proving existentence. So, we use ``max-min'' property to define Angelesco condition.
Next, we make a few short remarks on Angelesco condition meaning specifically
conditions eqref{8} - \eqref{9}of theorem \ref{thm1}.


It would be natural to suggest that the Angelesco condition depends
only on the mutual location of branch sets $e_k = e(f_k)$ of (actual)
singular points of components functions $f_k$ but in general it is not
true.  it is, probably, true in a situation of ``common position''.
However, some sufficient conditions may be stated in terms of vector
set $\vec e = (e_1, \dots, e_s)$.

In terms of branch sets, to infer that condition \eqref{9} holds it is enough to assume that component sets $e_j$ are in certain sense well separated. In more precise terms, let $\hat e_k$ be a convex hull of $e_k$. It is possible to prove that \eqref{8} holds if distances between sets $\hat e_k$ are large enough compare to sizes of those sets. However, it would not be an easy task to give constructive estimates for what is ``enough''.

It is known that it is not enough to assume that sets $\hat e_k$ are
disjoint. It follows form the results in \cite{ApKuVa07} where the case $s
=2$ is considered with the assumption that each set $e_1, e_2$ consists of
two points  (and some other assumptions; we will mention more details on
the case in section 4). At the same time the case of Markov type functions
show that the Angelesco systems may have overlapping sets $\hat e_k$.
Anyway, characterization of Angelesco condition in geometric terms is an
interesting and difficult problem which is not the main concern of this
paper.

\subsubsection{Second kind Hermite--Pad\'e polynomials.}
Methods of this paper (with some modifications) can be  used to study the zero distribution theorem for second kind Hermite--Pad\'e polynomials under the same Angelesco condition. We introduce definitions in section 2.1 below where Markov type functions are considered. However, theorems similar to theorems \ref{thm1} - \ref{thm3} above can be proved only under some stronger assumption on separation of branch sets $e_k;$  distances between sets $\hat e_k$ have to be large enough compare to sizes of those sets. We will state the theorem and outline their proofs in the end on section 4.

\subsubsection{More general assumptions on degrees of polynomials.}
Polynomials $q_{n, k}$  in \eqref{2} above were subject to the condition  $q_{n,k }\in \P_n, \  k=0,1,\dots, s$. This condition may be easily generalized. Consider an arbitrary sequence of vectors $\vec d_n = (d_{n, 1}, \dots, d_{n, s})$ with natural components. For any $n$ there exists a sequence of polynomials  $q_{n,k }, \  k=0,1,\dots, s$ such that $\deg  q_{n, k} \leq d_{n, k}$
and \eqref{2} is valid with  $O(1/z^{ns +s})$  replaced by $O(1/z^N)$ where $N = d_{n, 1} + \dots +\deg  d_{n, s} +s$.

 Zero distribution of such polynomials is described by a vector equilibrium problems, which is a generalization of the problem in lemmas \ref{1} and \ref{2} for the case when total masses of components are arbitrary positive numbers (prescribed in advance). Definitions are modified as follows.

 Let $t>0$ and $\mc M^{ t}(F)$ be the set of all positive Borel measures $\mu$ on the compact set $F$ with total mass $|\mu| = \mu(F) =t$. For a fixed vector-compact set $\vec{F}=\(F_1,\dotsc,F_p\)\in\mc{F}$
and a vector $\vec {t} =  (t_1,\dots, t_s)$ with positive components (total masses on components of vector-compact) we define a family of vector-measures
$$
{\vec {\mc M}}^{\vec t}={\vec {\mc M}}^{\vec t}(\vec{F}) =
\left\{\vec{\mu}
=(\mu_1,\dotsc,\mu_p):\mu_j\in\mc {M}^{t_j}\(F_j\)\right\}.
$$
Associated modifications of equilibrium problems are straightforward;
class of measures is modified, but the energy functional \eqref{4} -
\eqref{6} remains the same.

The analogue of lemma \ref{1} is valid: there exists a unique vector measure
${\vec\lambda}_{\vec \Gamma}({\vec t})$ in class ${\vec {\mc M}}^{\vec t}
 (\vec{F})$ minimizing the vector energy \eqref{6}.


In the definition \eqref{7} of the equilibrium energy functional $\mc
E[\vec F]$ class ${\vec {\mc M}}(\vec{F})$ has to be replaced by ${\vec
{\mc M}}^{\vec t}(\vec{F})$. Then the analogue of lemma \ref{2}) is
valid: there exists a vector compact set $\vec \Gamma \in
\vec{\mc{F}}$ in the class $ \vec{\mc{F}}(\vec f)$ which maximize
equilibrium energy  $\mc E( {\vec\lambda}_{\vec \Gamma}({\vec t}) )$
over this class (this compact set will now depend on $\vec t$).

Now, we have the following cumulative theorem.

 \begin{thm} \label{thm4}
 Let $\vec d_n =(d_{n,1}, \dots, d_{n,s} )$ be a given sequence of vectors with natural components and $q_{n, k}$ be the associated sequence of Hermite Pad\'e polynomials \eqref{2} ($\deg q_{n, k} \leq d_{n, k})$. If the following condition is satisfied
 \begin{equation*}
 \frac 1n\  d_{n,k}\  \to \ t_k >0, \qquad \text{as}\quad n \to \infty
\end{equation*}
 then the assertions of theorems \ref{thm1} (see \eqref{10}) and theorem \ref{thm3} are valid with $\lambda_k$ representing components of the vector equilibrium measure  ${\vec\lambda}_{\vec \Gamma}({\vec t})$.
Moreover, assertion of theorem \ref{thm2} is valid with $w_k$ standing for equilibrium constants (see \eqref{13} below) associated with the vector equilibrium measure  ${\vec\lambda}_{\vec \Gamma}({\vec t})$.
\end{thm}

The proof of the theorem \ref{4} would not be any different from the
proofs of theorems \ref{1} - \ref{3}. We will restrict ourself to
proving the $t_k =1$ since this particular case is identical to the
general one in all the essentials and allow us significantly simplify
notations (notations associated with arbitrary $\vec t$ are rather
crouded). At the same time we note that in the section 3 we use the
generalized equilibrium problem as a technical tool in proving theorems
\ref{thm1} -  \ref{thm3}.

\subsubsection{Possible generalizations of class $\mc A$ and Angelesco conditions.}
It is also possible to generalize classes $\mc A$ of functions $f \in
\mc A(\ol \CC \setminus e)$. Instead of finite sets $e$ we can consider
set $e$ of capacity zero.  This would require some essential
modifications in part related to properties of extremal vector compact
$\vec\Gamma$ and we will not go in details. At the same time in the
``scalar'' considerations associated with $\GRS$-theorem from
\cite{GoRa87} where more general settings on exceptional sets do not
cause any additional problems we preserve assumptions of \cite{GoRa87}.

The Angelesco condition on $ \vec\Gamma$ may be relaxed. Essentially we
need that each components of the support $ \Gamma^1_j = \supp(
\lambda_j)$ of extremal vector measure $\vec \lambda_{\vec \Gamma}$ do
not intersect the other components of $\vec \Gamma$. Moreover,
assertions of the theorem remain valid if we assume
that each intersection $ \Gamma^1_i \cap  \Gamma_j $ is a finite set
for $i \ne j$ (eventually, we may allow for intersections of zero
capacity). Such generalizations will make proofs more difficult and we
do not attempt any of these.

Finally, it was mentioned above that Angelesco conditions may not be
 dropped without essential modification of the potential theoretic problems
 which represent zero limit distributions of Hermite--Pad\'e polynomials.
 It is believed that under general settings (no Angelesco condition) there
 is still an equilibrium problem describing such zero distribution.  Matrix
 $A$ and, in general, conditions on total masses has to be changed.
 The general setting are essentially known for $s = 2$ for second kind polynomials;
 see \cite{ApKuVa07},~\cite{Apt08}, \cite{Rak11} and also~\cite{RaSu13},
\cite{BuSu15},~\cite{Sue16} for Nikishin system.
For $s > 2$ the form of the matrix is known only for some particular cases.

An alternative way exists to describe weak asymptotics of
Hermite--Pad\'e polynomials in terms of an Abelian integral on an
algebraic Riemann surface. The main problem associated with this
approach is that some accessory parameters of this Riemann surface are
generally not known. A conjecture on determination of this Riemann
surface is briefly described in the end of this introduction. Some
further details related to the case may be found in section 4.

\subsection{Methods}

The proof of theorem  \ref{thm1} 
is contained in sections 3 of the paper. We outline briefly the methods used there. As usual, the starting point 
is the following  system of complex (nonhermittian) orthogonality relations
\begin{equation} \label{11}
 \oint_C R_n(z) z^k dz = \oint_C \( q_1 f_1 + \dots + q_sf_s\) z^k dz = 0, \quad k =0,1, \dots, ns + s -2
\end{equation}
where $C$ is a Jordan contour separating $\cup e_j$ from infinity.
Equations \eqref{11} are easily derived from interpolation condition
\eqref{2} using a standard procedure.

A general method of working with such relations is available in case $s
= 1$.
In this case construction of Hermite--Pad\'e approximations \eqref{1}
reduces to diagonal Pad\'e  approximations. Polynomials $q_n = q_{n,1} $
are denominators of diagonal Pad\'e approximants to a (single) function
element $f = f_1$ at infinity. Zero distribution for these polynomials has
been obtained by H.~Stahl in his fundamental papers
\cite{Sta85a}--\cite{Sta86b} where an
original method of working with complex orthogonality relation was created.

To this end we mention that there is a number of systems of so-called semiclassical Pad\'e
(and Hermite--Pad\'e) polynomials which can be studied by generalizations of
methods developed in the theory of classical orthogonal polynomials.
In particular, differential equations were used in many instances; see \cite{MaRaSu16} for further references. A general method based on matrix Riemann- Hilbert has been developed in last two
decades which may be used to study the strong asymptotics of Pad\'e
polynomials of multivalued functions;
see \cite{ApYa15} \cite{Nu86},~\cite{Ra93},~\cite{MaRaSu12}).


Stahl's method has been substantially generalized  by A.~A.~Gonchar and author
in~\cite{GoRa87}
for the case when $f(z)$ in a certain way depends on $n$ (case of variable
weight). The theorem 3 from~\cite{GoRa87} which we will call $\GRS$-theorem
describes zero distribution of complex orthogonal polynomials with variable
weights. There are many applications of the theorem (see e.g.
reviews~\cite{ApBuMaSu11},~\cite{Rak16} and recent papers
\cite{RaSu13}, \cite{Bus15},~\cite{Sue16}).
In particular, after some modifications $\GRS$-theorem may be applied to study Angelesco case of Hermite--Pad\'e polynomials and this is a main point of this paper. As an introduction to the complex case we consider the Angelesco case for Markov type functions in section 2. We use this simpler situations
to introduce some basic ideas of the method, in particular, reduction of a vector situations to weighted scalar ones.

\subsubsection{Reduction of vector cases to scalar ones with external field. \\
 Markov type functions.}
A reduction of asymptotics problems for Hermite--Pad\'e polynomials to
similar problems for scalar ($s=1$) orthogonal polynomials with
variable weights was first applied in the paper \cite{GoRa81} for
Markov type function (matrix equilibrium problem has been introduced in
this connection).
Such a reduction is based on a basic property of vector equilibrium distributions - each its component is a scalar equilibrium in the external field generated by all the other components. In \cite{GoRa81} this observation was used to study zero distribution for the Angelesco Hermite--Pad\'e polynomials associated with Markov type functions.
The systems of orthogonality relations for such polynomials are ordinary
Hermitian orthogonality relations (see~\cite{Sz75},~\cite{Wi69},~\cite{StTo92})
and this makes Markov case essentially
simpler then the case $\vec f \in \mc A$.
The associated equilibrium problems are also simpler in Markov case. It
is {\it a priory} known that measures characterizing zero limit
distribution are supported on the real line and max-min condition
\eqref{3} is automatically satisfied for them for potentials of measure
on real line.

At the same time Markov case is, formally speaking, not a particular
case of the complex case and has its own independent value. After the
original paper \cite{GoRa81} the case has been  studied for the in
\cite{GoRaSo97} in more general settings including certain combinations
of Angelesco and Nikishin's cases. In both papers the second kind
polynomials studied. We consider the Angelesco--Markov case in Section 2
for both first and second kind Hermite--Pad\'e approximations. Results
of the paper for first kind approximations are probably new.

\subsubsection{Reduction in complex case.}

In Section 3 we turn to the complex case which is significantly more
complicated and the approach to the problem is at some points
essentially  different. We will arrange for a reduction of the zero
distribution problem for Angelesco Hermite--Pad\'e polynomials to a
somewhat modified version of $\GRS$-theorem on the zero distribution of
weighted (scalar) complex orthogonal polynomials. We revisit two basic
lemmas in the proof $\GRS$-theorem in \cite{GoRa87}. Using these lemmas we
prove a lemma (lemma \ref{5}, sec. 3 below) on asymptotics of certain
integrals of polynomials and this theorem is then used for the
reduction of vector case to scalar case with external field. This part
of the proof uses mostly techniques developed in \cite{GoRa87}.

The other part of the reduction process is related to the equilibrium problems for vector potentials.

\subsubsection{Equilibrium conditions. Vector $S$-equilibrium problems.}
Vector - equilibrium measure $\vec \lambda = \vec \lambda_{\vec F}$ for
an arbitrary $\vec F \in \vec{\mc F}$ is uniquely defined by the
following characteristic properly  $W_k(z) = w_k = \min_{F_k}W_k $ for
$z \in \supp( \lambda_k)$ where $W_k$ are components of associated
vector potential which may be represented in two equivalent ways as
follows.
\begin{equation} \label{12}
\begin{aligned}
& W_k(z) =\sum^s_{i=1} a_{i,k}U^{\lambda_i}(z) = U^{\lambda_k + \lambda}(z),  \qquad
\text{where} \quad \lambda = \lambda _1 +\dots+\lambda_s \\
& W_k(z) =2 \left( U^{\lambda_k}(z) +  \varphi_k(z) \right) 
\qquad \text{where} \quad  \varphi_k(z) =  \frac 12 \sum^s_{i\ne k } U^{\lambda_i}
\end{aligned}
\end{equation}

Also, for an arbitrary $\vec F \in \vec{\mc F}$ corresponding vector -
equilibrium measure $\vec \lambda = \vec \lambda_{\vec F}$ is uniquely
defined by the equilibrium conditions. We present this conditions
in case $\vec F = \vec \Gamma$.
\begin{equation} \label{13}
W_k(z) = w_k, \quad z \in \Gamma^1_k = \supp( \lambda_k), \qquad
W_k(z) \geq w_k, \quad z \in \Gamma_k
\end{equation}
Thus, these equilibrium conditions are valid for any $\vec F \in \vec{\mc F}$
(with $\vec \lambda = \vec\lambda_{\vec F}$). The extremal vector compact set $\vec \Gamma$ has in addition another important symmetry property.

The following {\it vector $S$-property} (for symmetry) is valid for the equilibrium potentials
associated with the extremal compact set $\vec \Gamma$ under the Angelesco
condition. For any $j = 1, \dots, s$
 there exist a finite set $e_j$ such that  $\Gamma^0_j =  \Gamma^1_j \setminus  e_j$ is a disjoint finite union of open analytic arcs and
\begin{equation} \label{14}
\frac{\partial W_j}{\partial n_1}(\zeta)
=\frac{\partial W_j}{\partial n_2}(\zeta),\quad
\zeta\in\ \Gamma^0_j ;\quad
\end{equation}
where $n_1,n_2$ are opposite normals to $\supp(\lambda_j)$ at  $\zeta$.
Actually it is the $S$-property what we need to study the complex
orthogonal polynomials. The max-min energy problem is a general way to
prove existence of such a compact.

More exactly, solutions of asymptotics problems related to complex
orthogonal polynomials are often reduced to an existence problem for
compact set with $S$-property in given class of compact sets ($S$ equilibrium
problem). In many cases such a problem may be solved by reduction to
the problem maximization of equilibrium energy in given class. This is
exactly the method we use in the paper. As a remark we note that in
general settings vector $S$-equilibrium problems it rather difficult to
solve using this method. Under the Angelesco condition the reduction is
simpler. We will actually reduce the vector problem to weighted scalar
one and use methods developed in \cite{Rak12} to study a weighted scalar
problems (see also~\cite{MaRaSu11},~\cite{BuMaSu12}).

To this end, we emphasize an important difference between Markov (real) and general (complex)
cases for Hermite--Pad\'e polynomials is exactly in the structure of electrostatics associated with the situation. For Markov type functions associated vector equilibrium problem does not include the second part which is related to the maximization of the equilibrium energy \eqref{8} in lemma \ref{2}. Segments of real axis are exactly the curves maximizing the equilibrium energy in classes of continuums with fixed end points. Equivalently, the $S$-property \eqref{14} is valid for potentials of any measures on real line, since potentials of such measures are symmetric with respect to real line. So, the $S$-problem is not there.

In complex case the $S$-property of the extremal compact set $\Gamma$ in
the problem \eqref{8} is the substitution for such symmetry and it not
available {\it a priory}. Finding compact sets with $S$-property is, then,
an important part of the problem. The $S$-property plays also an
important role in construction of a Riemann surface $\mathcal R =
\mathcal R_{\vec f} $ which will introduced in section 4 of the paper.
In this section we  go first into some technical details regarding
properties of $\vec \Gamma$. In particular, we prove existence (lemma
\ref{2}) and some continuity properties of extremal compact sets. Methods
used here are mostly developed in \cite{Rak12}. Then we turn to the
Riemann surface which is technically used to establish some properties
of $\vec \Gamma$. At the same time it is useful in a discussion of the
problem at large.




\subsubsection{Riemann surface.}

There is an old observation going back to 1970-th and 80-th
that for some classes of vector - functions $\vec f = (f_1,\dots,
f_s)\in\mathcal A$ there is a Riemann surface  which in a sense controls
asymptotics of corresponding Hermite--Pad\'e polynomials. The fist instance
of such a situation was reported in V.~A.~Kalyagin's paper \cite{Ka79}.
Riemann surfaces also play a central role in the J.~Nuttall's important
review \cite{Nu84} (see also~\cite{Nu81}). Independently, for Markov Angelesco case construction
of such a surface from the solution of the vector-equilibrium problem was
presented by A.~I.~Aptekarev and V.~A.~Kalyagin in \cite{ApKa86}. They also
noticed that the connection goes both ways and a vector equilibrium problem
could be recovered from known Riemann surface. Therefore, the Riemann
surface may play a role similar to the one played by the vector equilibrium
measure $\vec \lambda(\vec f)$; see also~\cite{GoRaSu91},~\cite{GoRaSu92}.

With a time it become a commonly excepted conjecture that for arbitrary
$\vec f = (f_1,\dots, f_s)\in\mathcal A$ there exists a Riemann surface
$\mathcal R = \mathcal R_{\vec f} $  of $s+1$ sheets such that
asymptotics of Hermite--Pad\'e polynomials associated with $\vec f$ may
be described in terms of special functions which belong to this
surface. This Riemann surface may present a general approach to the
asymptotics of the Hermite--Pad\'e polynomials and it is an alternative
to the vector equilibrium problem. It is believed that the two ways are
closely related and, may be, formally equivalent.

One of the main problems related to this conjecture is that in general
$\mathcal R$ is not known. Another problem, not yet generally solved,
is the exact way to pass from the Riemann surface to asymptotics. Note
that the structure of the vector equilibrium problem is in general not
known either. But the Riemann surface still has an advantage; it
contains only a finite number of unknown parameters while vector
equilibrium problem may have an unknown geometric structure.

To this end we present a conjecture (due to
A.~Mart\'{\i}nez-Finkelshtein, S.~P.~Suetin and
the author) which was formulated on a basis of partial results from
\cite{MaRaSu16} on zero distribution of first kind Hermite--Pad\'e
polynomials for functions $\vec f = (f_1,\dots, f_s)\in\mathscr L$.
Such polynomials satisfy a linear differential equation with polynomial
coefficients which gives a powerful approach to asymptotics. We
formulate a more constructive form of the conjecture on existence of a
Riemann surface  $\mc R = \mc R(\vec f)$ including a method which may
help to answer the key question: how to determine its unknown
parameters.






\section{Markov type functions. }

In this section we revisit the case of Markov type functions which is a model situation where the Angelesco condition is most transparent.

The first result on asymptotics of the Hermite--Pad\'e polynomials has
been obtained  by  V.~Kalyagin \cite{Ka79} for the case of two
Jacobi-type functions whose weights had nonoverlapping but adjoint
supports. He found a generating function for the second kind
Hermite--Pad\'e polynomials and, then, used a generalization of the
classical Darboux method to study strong asymptotics of polynomials:
in this connection see also ~\cite{Sz75}, \cite{NiSo88},
\cite{ApMaRo97} and references therein. Analgebraic Riemann surface
has been for the first time introduced in  \cite{Ka79}.


A potential theoretic approach to the problem has been developed in the
paper  \cite{GoRa81} by A.~A.~Gonchar and the author for a vector $\vec f
= (f_1, \dots, f_s) $ of Markov type functions. In the  Angelesco
situation a theorem on zero distribution of the Hermite--Pad\'e
polynomials has been proved in   \cite{GoRa81} and a vector equilibrium
problem has been used to characterize the limit zero distribution.
Methods of the current paper are in part originated in \cite{GoRa81}.
We will recall some details related to the case.

Markov type function is a Cauchy transform of a positive measure on
$\R$ with compact support
\begin{equation} \label{M1}
f(z)\  =\  \int \ \frac { d\sigma (t)}{z-t}, \quad z \in \overline{\mathbb C}\setminus \supp(\sigma)
\end{equation}



The Angelesco case for Markov type functions is defined by the
condition that supports of associated measures are disjoint or, at
least, not overlapping. We do not pursue maximal generality in this
discussion, so, we will assume that supports of measures are disjoint
intervals, measures are absolutely continuous and their densities are
positive $a.e$. on corresponding intervals. Thus, we consider vector
function $\vec f = (f_1, \dots, f_s)$ where
\begin{equation} \label{M2}
f_k(z)\  =\  \int \ \frac { w_k (t) }{z-t} \ dt , \quad z \in \overline C \setminus F_k, \quad F_k =  \supp(\sigma_k), \quad k =1, 2, \dots, s
\end{equation}
and
\begin{equation} \label{M3}
 \quad w_k(x) >0 \ \ a.e. \text{ on the interval } \ \ F_k,
 \qquad k =1, 2, \dots, s; \qquad F_i \cap F_j = \varnothing, \ \ i \ne j.
\end{equation}
We note that there is in a sense opposite case when two Markov
functions are generated by measures on the same interval $F =
\supp(\sigma_1) = \supp(\sigma_2)$. The case was introduced by
E.~M.~Nikishin~\cite{Nik86} (see also~\cite{NiSo88}) and was named after him.

\subsection{Second kind Hermite Pad\'e polynomials.}\label{subsec2.1}

There two types of Hermite Pad\'e approximations and polynomials which are called polynomials of the first (or Latin) and, respectively, second (or German) kind. Both kinds are associated with a vector $\vec f =(f_0, f_1, \dots, f_s )$ of $s\in \N$ analytic functions defined by their Laurent expansion at infinity
\begin{equation}
f_k(z)= \sum_{m=0}^\infty \frac{f_{m,k}}{z^m}, \quad k=0,1,\dots.
\label{M4}
\end{equation}
First kind polynomials $Q_{n,k}$ were introduced in \eqref{2} above
(for the case of equal degrees of polynomials). We present a definition
of the second kind polynomials and approximations for the similar case
when order of approximation is same for each the function. For a
natural $n\in\N$ the $n$-th vector of the second  kind {\it
Hermite--Pad\'e approximations }
$\pi_{n,k}(z) = \wt P_{n,k}/P_{n}$
with a common denominator $P_n(z)$ is defined by the following
relations
\begin{equation} \label{M5}
P_n(z) f_k(z) - \wt P_{n, k} = O\(\frac1{z^{n+1}}\), \quad k =1, 2, \dots, s; \qquad P_n \in\PP_{ns}
\end{equation}
Clear that degrees of all numerator polynomials $\wt P_{n, k}$ are at most ${ns }$. 

Hermite--Pad\'e polynomials of both kinds are equivalently defined by
certain systems of orthogonality relations. In case of Angelesco Markov
type functions the orthogonality involved is the usual (traditional)
Hermitian orthogonality with positive weights which is generalized it
two different ways.

Orthogonality conditions \eqref{11} for the remainder associated with the
first kind approximations in Markov case \eqref{M2}--\eqref{M3} take form
\begin{equation} \label{M6}
\begin{aligned}
& \oint_C R_n(z) z^j dz = \oint_C \( q_{n,1} f_{n,1} + \dots + q_{n,s}f_s\)(z) z^j dz = \\
&=\int_{F_1}   q_{n,1}(x) w_1 (x) x^j dx+ \dots + \int_{F_s} q_{n,s}(x) w_s(x) x^j dx = 0
\end{aligned}
\end{equation}
for $ j =0,1, \dots, ns + s -2$. This may be equivalently written as
$$
\int_\R r_n(x) g(x) dx = 0
$$
for any polynomial $g \in\PP_{ns+s -2}$ where
the function $r_n(x) = (q_{n,1} w_1+ \dots +  q_{n,s} w_s) (x) $ is
defined on the whole real axis if we assume that $w_k(x) = 0$ in the
complement to $F_k$ It follows, in particular that the function $r_n$ has
at least $ns + s -1$ sign changes on $\R$. For $x \in F_k$ the function
$r_n(x) = q_{n,k}w_k$ changes sign at most $n$ times. It follows that
$q_{n,k}$ has exactly $n$ simple zeros in $F_k$. In addition $r_n$
changes sign every time when we pass from $F_k$ to $F_{k+1}.$

Orthogonality conditions for the second kind polynomials common denominator
$P_n$ in \eqref{M2} are
\begin{equation} \label{M7}
\int_{F_k}   P_n w_k x^j dx = 0,  \quad j =0,1, \dots, n - 1, \quad  k = 1, \dots, s.
\end{equation}

In particular, it follows from here that polynomial $P_n$ in \eqref{M5} has exactly $n$ simple zeros on each interval $F_k$  and, therefore admit factorization $P_n = \prod_{k=1}^s P_{n,k}$ where all polynomials $P_{n,k}$ are of degree $n$ and zeros of $P_{n,k}$ belong to $F_k$.

\subsection{Theorem on zero distribution.}

Under assumptions above the limit distributions os Hermite--Pad\'e
polynomials are  characterized by the following theorem.

\begin{thm} \label{M-1}
We have
\begin{equation}\label{M-11}
\begin{aligned}
&(A) \quad \nu_{n} = \frac 1n \sum_{P_{n}(x) = 0}\delta (x) \ \overset{*}{\to} \
\lambda_1 + \lambda_2 +\dots + \lambda_s\\
&(B) \quad \mu_{n,k} = \frac 1n \sum_{q_{n,k}(x) = 0}\delta (x) \ \overset{*}{\to}\ {\lambda_k},\quad
k =1,, \dots, s.
\end{aligned}
\end{equation}
If $R_n$ is normalized by the conditions $c_{n,1} = 1$ then
other leading coefficients are positive and we have
\begin{equation} \label{M-12}
\lim_{n\to \infty} \frac 1n \log |c_{n,k}| = w_k - w_1
\end{equation}
where $\vec{\lambda} = \vec{\lambda}_{\vec F}\in\vec{\mc{M}} (\vec{F})$ is vector equilibrium measure from lemma \ref{lem1} and $w_k$ are equilibrium constants in \eqref{12} - \eqref{13}
\end{thm}

Thus, zero distribution of Hermite--Pad\'e polynomials of both first
and second kind are represented by the same vector equilibrium measure
$\vec \lambda$ (see lemma \ref{lem1}). In other terms, it follows from
the theorem \ref{M-1}  that the second kind polynomial $P_n \in
\P_{ns}$ has the same  the $n$-th asymptotics as the product
$q_{n,1}\dots q_{n,s}$ of first kind polynomials (up to a normalizing
constant).

We note also that in case of Markov type functions the lemma \ref{lem2} is not involved in the definition of limit measures $\lambda_k$. The reason is the symmetry of potentials of measures $\mu$ supported on $\R$ with respect to $\R$: we have $U^\mu(\overline z) = U^\mu(z)$ for any such measure and for any $z$ not in the support of $\mu$. It follows that the $S$-property \eqref{14} is satisfied for $U^\mu$ at points where normal derivatives exist. Hence,  in the Markov case we have $\vec\Gamma = \vec F$. 

An important difference between the first and second kind polynomials is
that leading coefficients of the first kind polynomials are defined by
orthogonality conditions \eqref{M6} like all other coefficients. More
exactly, the function $r_n$ is defined by \eqref{M6} up to a multiplier, so
that the leading coefficient $c_{n,k}$ of one of the polynomials
$q_{n,k}(x) =c_{n,k}Q_{n,k}(x)$ may be selected arbitrarily as a
normalization condition. Then the other coefficients are uniquely
determined by the orthogonality conditions \eqref{M6}. Asymptotics of
leading coefficients presented in \eqref{M-12} is an important part of the
solution of the problem. In particular, we need this asymptotics to obtain
the asymptotics of $|q_{n,k}(z)|^{1/n}$ in $\CC \setminus F$ as $n \to
\infty$.

As it was mentioned above, the part $(A)$ of the theorem is well known; see
\cite{GoRa81} and \cite{GoRaSo97}. Part (B) seems to be new. To this end we
note that most part of the papers in the literature on asymptotics of
Hermite Pad\'e polynomials were devoted so far to the second kind
approximations and polynomials.

In the connection to the Markov case see also some more recent results
in  \cite{ApLy10} and \cite{Rak11} where mixed case has been studied for
two Markov functions in case when one support is a subinterval of
another one. Nikishin case is considered in \cite{LoPe15},~\cite{ApLoMa17}; see
also~\cite{RaSu13},~\cite{Sue15},~\cite{Sue17}.

We will present a proof the theorem \ref{M-1} based on a standard
reduction of a vector orthogonality to a scalar weighted orthogonality.
In terms of associated equilibrium problems this is equivalent to a
possibility to define vector equilibrium in terms of scalar weighted
equilibrium. All these ideas are essentially well known, but we will
brief but connected presentation anyway. First, some of the results are
new. Second, we use the opportunity to introduce in a comparatively
simple situation certain arguments which will be generalized and used
in complex case.


\subsection{Equilibrium measure in an external field.}\label{2.3}
We begin with necessary definitions.

Let $F \in \mc F$; recall that such compact sets are regular with respect
to the Dirichlet problem.  By an external field on $F$ we will
understand generally a continuous real valued function  $\varphi(x)$ on
$F$ if it is not explicitly said otherwise. Class of external field may
be significantly generalized (see e.g. \cite{SaTo99})  but we want to
keep  this exposition as short as possible.


As above, we use notation $\mc M$ for all positive Borel  measures in plane.  Let $ t  > 0$ and
$\mc M^t(F) $ be a set of $\mu \in \mc M$ on $F$ with total mass $\mu(F) = t.$
By $\mc{E}_\varphi(\mu)$ we denote the (total) energy of a measure $\mu \in \mc M$ in the external field $\varphi$
\begin{equation}
\label{M15}
\mc{E}_\varphi(\mu)=\iint\log\frac1{|x-y|}\,
d\mu(x)\,d\mu(y)+2\int\varphi(x)\,d\mu(x),
\end{equation}
The equilibrium measure $\lambda=\lambda^t_{\varphi, F}$ for a compact
$F$ in the external field $\varphi$ is defined by the following
minimization properties
\begin{equation}
\label{M16}
\lambda \in \mc M^t(F), \qquad
\mc{E}_\varphi(\lambda)=\min_{\mc{M}\in\mc M^t(F)}\mc{E}_\varphi(\mu).
\end{equation}
Equivalently the equilibrium measure $\lambda^t=\lambda^t_{\varphi, F}$ is defined by the following equilibrium conditions for total potential
\begin{equation}
\label{M17}
\(V^\lambda+\varphi\)(x)
\begin{aligned}
& = w, \quad  
 ~ x\in\supp\lambda, \\
&\ge w, \quad ~x\in F.
\end{aligned}
\end{equation}
Equation \eqref{M17} uniquely defines the pair of measure $\lambda\in\mc{M}(\Gamma)$ and constant $w=w_{\varphi, F}$ -- equilibrium constant. In case $t =1$ (the main case in this paper) we drop index $t$ from notations.

We note that for  more general classes of external fields $\varphi$ and/or for more general classes of compact set $F$ in plane equality in the first line in \eqref{M17} holds except for a set of capacity zero (inequality $\leq$ holds at any point). In the connection with these definitions see original papers \cite{GoRa84} \cite{GoRa87} and the book \cite{SaTo99}.

\subsection{Orthogonal polynomials on $\R$ with varying weight.}

Let $F$ be a finite union of intervals, $\Phi_n(x)$ be a sequence of positive, continuous real-valued functions on a set $F$ sa
\begin{equation}\label{Field}
\varphi_n(x) : =  \frac 1{2n}\ \log \frac 1{\Phi_n(x)}\ \to \ \   \varphi(x)
\end{equation}
uniformly on $F$ (so that $\varphi$ is a continuous function on $F$)
and  $f(x)>0$ a.e.\  on $\Gamma$. Let polynomials $Q_n(x)=x^n+\dotsc$
are defined by orthogonality relations
\begin{equation}
\label{Ort}
\int_F Q_n(x)\ x^k~ \Phi_n(x) f(x)\,dx=0,
\quad k=0,1,\dotsc,n-1.
\end{equation}
The sequence $Q_n$ presents a typical example of what is called orthogonal
polynomials with varying weight; it usually means that the weight functions
depends on the degree $n$ of the polynomial in such a way that their $n$-th
root asymptotics exists (clear that \eqref {Field} remains essentially
valid if we replace there $\Phi_n(x)$ with the total weight $\Phi_n(x)
f(x)$). The following theorem by A.~A.~Gonchar and the author \cite{GoRa84}
was the first general result on zero distribution of such orthogonal
polynomials.  We present a simplified version needed for our purposes.

\begin{thm}\label{M-3} Under the assumptions on $F, \Phi_n, f$ stated above, we have
$$
\frac1n\  \mc{X}\(Q_n\)=\frac{1}{n}\sum_{Q_n(\zeta)=0}
\delta(\zeta)\ \  {\overset{*}{\to}}\ \  \lambda
$$
where $\lambda = \lambda_\varphi$ is the equilibrium measure on $F$ in the external field $\varphi$.
\end{thm}

The theorem has many applications; here we use it for the reduction of a vector zero distribution problem to a scalar one.

The original proof in \cite{GoRa84} was based on the $L^2$ extremal property of polynomials $Q_n$.
We present a different proof based directly on the orthogonality conditions. This proof exhibits in a simple situation one of a basic elements of the method which we will use also in case of complex valued weights.

\begin{proof}
From the contrary, assume that assertion of the theorem is not true. Then it is possible, using weak-star compactness of the space of unit measures (on the sphere), to select a weakly convergent subsequence $\quad \frac1n\  \mc{X}\(Q_n\)  \to  \mu \quad$ where
 $\Lambda= \{n_k\}_{k=1}^\infty \subset \N$  such  $\mu \ne \lambda.$

Then a contradiction with the orthogonality relations \eqref{Ort} will
be obtained by construction a sequence of  a polynomial $P_n,\quad n\in
\Lambda$ of degree $<n$ for which the $n$-th root asymptotics of the
expression
\begin{equation*}\label{}
 I_n(F) = \int _{F} Q_n(x) P_n(x)\Phi_n(x) f(x) dx
\end{equation*}
may be determined. The asymptotics will infer that $I_n(f) \ne 0$ for large enough $n$ which presents desired contradiction.

We construct the desired polynomial $P_n$ by duplicating the polynomial $Q_n$
but making one small modification. More exactly, we need that $\deg P_n
<n$ so that we have to drop at least one zero of $Q_n.$ We will drop
two zeros selected as follows. Since $\mu \ne \lambda$ the equilibrium
conditions \eqref{M17} are not satisfied. It follows that there is a
point $x_0 \in \supp \mu$ such that $(U^\mu+\varphi)(x_0) > m = \min_{x
\in F}( U^\mu+\varphi)(x)$. Taking into account lower discontinuity of
$U^\mu$ we infer that there is $r$-neighborhood $\Delta = [x_0-r, x_0 +r]$,
$r
>0$ such that $( U^\mu+\varphi)(x) \geq m_1 > m $ for $x \in \Delta
\cap F.$

Since $x_0 \in \supp \mu$ we can find for large enough $n$ two zeros
$a_n, \ b_n$ of $Q_n$ in the $r$-neighborhood of $x_0$ and we
define
\begin{equation}\label{P_n}
P_n(x) = \frac {Q_n(x)}{(x-a_n)(x-b_n)}
\end{equation}
With this $P_n$ we will determine asymptotics of $I_n(F)$ above. We have
$$
 \mc{X}_n =  \frac1n\  \mc{X}\(Q_n P_n\) 
\ \  {\overset{*}{\to}}\ \  2 \mu
$$
Now we use well known convergence properties of potentials of arbitrary
weakly convergent sequence of measures, say, $\nu_n \ \to  \nu$. This
implies that $U^{\nu_n} \to U^\nu$ in linear Lebesgue measure on any
rectifiable curve in plane. Convergence of potentials is also
semiuniform from below. In particular we have convergence of minimums
$\min_F U^{\nu_n} \to \min_F U^\nu$ over any regular (for the Dirichlet
problem) compact set $F$ in plane (see \cite {Lan66}). It follows that we
have convergence (of nonnegative functions)
$$
E_n(x) = \exp \left\{ - U^{\mc X_n}(x) + \varphi_n (x)  \right\}   \to
E(x) = \exp \left\{ - 2 ( U^{\mu}(x) + \varphi (x)) \right\}
$$
on $F$ is in measure and semiuniform from above. Since $f >0$ a.e. on $F$ it follows the
$$
\lim_{n\to\infty}
|I_n(F \setminus \Delta)|^{1/n} =
\lim_{n\to\infty}
\left(\int _{F \setminus \Delta} E^n_n(x)  f dx\right)^{1/n}
= e^ {-2m} = \max_{F \setminus \Delta} E(x) > 0
$$
(note that $Q_n P_n \Phi_n f \geq 0$ a.e. on $F \setminus \Delta$). On the other hand we have
$$
\varlimsup_{n\to\infty}|I_n(F \cap \Delta)|^{1/n}\leq
\varlimsup_{n\to\infty}
\left(\int _{F\cap\Delta} E^n_n(x)  f dx\right)^{1/n}
= e^ {-2m_1} = \max_{F \cap \Delta} E(x) < e^{-2m}.
$$

  It follows that $\int_F Q_n P_n\Phi_n f dx \ne 0$ for large enough $n$ in contradiction with orthogonality conditions for $Q_n$
\end{proof}

\subsection{Proof of theorem \ref{M-1}}

Part (A) of theorem is well known, however, we present a proof. It is short and illustrate one important detail in the procedure of reduction of vector cases to weighted scalar ones.
\begin{proof}
We use orthogonality relations \eqref{M7}. As it is mentioned
afterwards, it follows by \eqref{M7} that $P_n = \prod_{k=1}^s P_{n,k}$
where all polynomials $P_{n,k}$ are monic, of degree $n$ and zeros of
$P_{n,k}$ belong to $F_k$. We note that here we {\it a priory} have
important information about location of zeros of polynomials. This is
crucial for the proof below;
see remarks in sec. 2.6 on the case of complex valued weights.

Select a  subsequence $\Lambda= \{n_k\}_{k=1}^\infty \subset \N$  such that
$$\mu_{n,k } = \quad \frac1n\  \mc{X}\(P_{n, k}\)  \to  \mu_k, \quad \text{as}\quad
n\to \infty, \ n \in \Lambda; \qquad k = 1, 2, \dots, s. $$
Now it is enough to prove that $\mu_k  = \lambda_k$.

We will show that for any $k$ the measure $\mu_k$ is the unit
equilibrium measure in the external field $\varphi_k(x) = \frac 12
\sum_{j \ne k}U^{\mu_j}(x).$  Indeed,  the group of orthogonality
conditions in \eqref{M7} associated with the index $k$ defines
$P_{n,k}$ as the orthogonal polynomial on $F_k$ with varying weight
$\Phi_n(x) w_k(x)$ where $\Phi_{n,k}(x) = \prod_{j \ne k}P_{n,j}(x).$
For a fixed $k$ this function satisfy condition \eqref{Ort} with
$\varphi(x) = \varphi_k(x)$ and we desired proposition follows by
theorem \ref{M-3}.  In terms of the energy functional
$\mc{E}(\vec{\mu})$ in lemma \ref{lem1} this proposition means that
vector measure $\vec{\mu} = (\mu_1, \dots, \mu_s) $ provides a
component wise minimum of this functional, that is, minimum of  $\mc{E}
(\mu_1, \dots, \mu, \dots, \mu_s) $ over $\mu \in M(F_k)$ (variable
measure $\mu$ takes place of $\mu_k$) is achieved when $\mu = \mu_k$.

 The energy functional $\mc{E}(\vec{\mu})$ is convex
(see \cite[Capter 5, Lemma 4.3]{NiSo88}) and, therefore, does not have componentwise minimuma different from the global one.
\end{proof}
\vspace{5mm}

The main difference between the proof of the part B of the theorem and
the proof of part A above is that now asymptotics \eqref{M-12} of the
leading coefficients $ c_{n,k}$ of polynomials $q_{n,k}(x) =  c_{n,k}
x^n + \dots $ is proved simultaneously with zero distribution.
\begin{proof}

 We select a subsequence $\Lambda= \{n_k\}_{k=1}^\infty \subset \N$  such that
\begin{equation*}
\mu_{n,k } = \quad \frac1n\  \mc{X}\(q_{n, k}\)  \to  \mu_k, \quad \text{as}\quad
n\to \infty, \ n \in \Lambda; \qquad k = 1, 2, \dots, s.
\end{equation*}
and also
\begin{equation*}
|c_{n,k}|^{1/n} \to e^{- u_k} \quad \text{as}\quad
n\to \infty, \ n \in \Lambda; \qquad k = 1, 2, \dots, s.
\end{equation*}
where $u_k  \in [-\infty, +\infty].$ We have to take into account a
possibility that some of numbers $u_k$ may not be finite. In this
connection it is convenient to normalize function $r_n(x) $
such that $\max_k u_k = 0$, so that $u_k \leq 0$ for all $k.$ Then all numbers $e^{-u_k}$ are finite.

Orthogonality conditions \eqref{M6} are equivalent to the assertion that for any polynomial $P_n \in\PP_{ns +s -2 }$ we have
\begin{equation*}\label{}
\sum_{j=1}^s \int_{F_j}q_{n,j}(x) P_n(x) w_k(x) dx = \int_F r_n(x) P_n(x) x = 0
\end{equation*}
where \ $r_n(x) = q_{n,k}(x) w_k(x)$ on $F_k$ and  $r_n(x) = 0$ for $x \notin F =\cup F_k.$


Now we will construct a polynomial $P_n \in\PP_{ns +s -2 } $ for which it is not valid if assertions of the part (B) of the theorem are violated. A preliminary definition of $P_n$ contains a few parameters which will be determined later. We select one of the integers $k$ from $1 \leq k \leq s$ (let us call this index $k$ exceptional). Then we select any two zeros $a_n, b_n$ of the polynomial $q_{n,k}$
and define
\begin{equation}\label{P_n-1}
P_n(x) =  \frac{ h_n(x)}{ (x- a_n)(x- b_n) }\   \prod \ c^{-1}_{n,k}\  q_{n,k}(x)
\end{equation}
where $h_n(x)$ is a monic polynomial of degree $s-1$ which has one simple zero at each gap between intervals $F_k$. Thus, $P_n(x) \in\PP_{ns +s -3}$ is a sequence of monic polynomial with the limit zero distribution not depending on the choices of parameters. We have  as $n\to \infty, \ n \in \Lambda$
$$
\frac1n\  \mc{X}\(P_{n}\)  \to \mu = \sum_{k=1}^s \mu_k, \quad \text{and, therefore,}\quad
\frac1n\  \mc{X}\(q_{n,j} P_{n}\)  \to \mu_j + \mu, \quad j = 1, \dots, s
 $$

Next, we define
$$
I_{n.j} = \int_{F_j} r_n(x) P_n(x) dx  =
\int_{F_j} q_{n,j}(x)w_j(x) P_n(x) dx
 \qquad \text{for}\quad \quad j = 1, \dots, s.
$$
It follows from the construction of polynomials $P_n$ that the function $r_n(x) P_n(x)$ is nonnegative on $F \setminus [a_n, b_n]$ and, therefore, $r_n(x) P_n(x) = ( q_{n, j}w_jP_{n})(x)\geq 0$ on $F_j$ for any non exceptional $j$; from here $I_{n.j} > 0$ for $j \ne k.$

Finally, using the same convergence properties of potentials as in the
proof of part A we conclude that as $n\to \infty$ and  $n\in \Lambda$
we have the following relations
\begin{equation} \label{I}
I^{1/n}_{n,j}  \to e^{-m_j}  \quad \text{for}\quad \quad j \ne k \qquad
\text{and} \quad  \overline \lim\  |I_{n,k}|^{1/n}  \leq  e^{-m_k}
\end{equation}
as $n\to \infty, \ n \in \Lambda$ where
\begin{equation} \label{I-1}
m_j :=  u_j + \min_{x \in F_j}\ U^{\mu_j +\mu}(x), \qquad  j =1, \dots, s.
\end{equation}

Now we prove that all numbers $m_j$ are equal for $j = 1, \dots , s$.
If it was not so then the minimal of them will be strictly less the
maximal one, that is, there is $k$ such that  $m_k  < m = \max_{j}
m_j.$ For this exceptional index $k$ we select as $a_n, b_n$ any two
zeros of $P_n$ in $F_k$ and choose arbitrarily  other parameters of
$P_n.$ Orthogonality of $r_n$ to this polynomial in terms of constants
$I_{n,j}$ is written as
$$
\sum_{j=1}^s I_{n,j}= 0, \qquad\text {so that} \qquad |I_{n,k}| = I_n:= \sum_{j\ne k}^s I_{n,j}
$$
It would imply $m_k = m$   in contradiction with the inequality $m_k < m$ established above.

It remains to prove that $\vec \mu =\vec \lambda$. We use again an approach used in the proof of part (A), that is, we observe that it is enough to prove that each component $\mu_k$ of $\vec \mu$ is the equilibrium measure in the field $\varphi_k(x) = \frac 12 U^{\nu_k} (x)$ where $\nu_k = \nu - \mu_k = \sum_{j\ne k} \mu_j.$ However, here we can not make a direct reference to the theorem \ref{M-3} as we did when proving part (A); instead we will apply the method which we used in the proof of this theorem.

Assume the contrary; one of the components, say $\mu_k$, is not an
equilibrium measure in the field $\varphi_k(x).$ Without loss of
generality we may assume that $ k =1.$ Thus, equilibrium condition is
violated for $\mu_1 $, which means that we can find a point $x_0 \in
\supp \mu_1$ such that $(U^{\mu_1} + \varphi_1)(x_0) > m_1 = \min_{x
\in F}( U^{\mu_1}+\varphi_1)(x)$. Then we can also find a neighborhood
$\Delta = [x_0-r, x_0 +r],\ r >0$ such that $( U^{\mu_1}+\varphi_1)(x)
\geq m_0 > m_1 $ for $x \in \Delta \cap F.$

It follows from  $x_0 \in \supp \mu_1$ that for large enough $n$ there exist two zeros $a_n, \ b_n$ of $q_{n,1}$ in  $\Delta.$ Using these parameters we define polynomial  $P_n(x)$ by \eqref{P_n-1}. The conclusion of the proof is similar to what we used to prove equality of constants $m_j$. The small modification has to be made: in place of exceptional set $F_k$ (now $k=1$) we separate a part of it, namely, the segment $\Delta_n = [a_n,b_n] \subset \Delta \subset F_1$. We have $r_n(x)P_n(x) \geq 0$ on $F \setminus \Delta_n$ and orthogonality condition \eqref{M6} imply
$$
\int_{F_1 \setminus\Delta}r_nP_n dx \leq  \int_{F \setminus\Delta_n}r_nP_n dx  = - \int_{\Delta_n}r_nP_n dx   =
 \left|\int_{\Delta_n}r_nP_n dx \right| \leq \int_{\Delta}|r_nP_n| dx.
$$
Now we can raise both pats to power $1/n$ and go to the limit as $n\to \infty, \quad n \in \Lambda.$ Limit to the left will be strictly larger then limit to the right in contradiction with inequality above (the conclusion is identical to that in theorem \ref{M-3}).
\end{proof}

\subsection{Remarks on generalization for complex weights}
In the next section 3 we will study zero distributions of Hermite
Pad\'e polynomials for functions $\vec f \in \mc A.$ In equivalent
terms it means that we will deal with systems of complex orthogonality
relations in place of Hermitian ones. This would require essential
modifications in the method. At the same time, the method at large and
certain technical elements will be similar to what was used above in
Markov case. To illustrate similarities and differences between the two
cases in more specific terms we make a few remarks related to the case
of Markov type functions with complex valued weights. Such a situations
may be considered as transitional between the two case indicated above.

Let $\vec f = (f_1, \dots, f_s)$ be a vector function whose components
are defined as Markov-type functions
\begin{equation} \label{M3-1}
f_k(z)\  =\  \int \ \frac { w_k (t) }{z-t} \ dt , \quad z \in
\overline\CC \setminus F_k, \quad F_k =  \supp(\sigma_k), \quad k =1, 2, \dots, s,
\end{equation}
with complex valued weights $w_k(x)$ on disjoint real sets $F_k$, each
of them is a finite union of intervals, with the following conditions
\begin{equation} \label{M3-2}
|w_k(x)| >0 \ \ a.e. \text{ on the interval } \ \ F_k,
 \qquad k =1, 2, \dots, s; \qquad F_i \cap F_j = \varnothing, \ \ i \ne j.
\end{equation}
and there is a closed set $e_k \subset F_k$ such that
\begin{equation} \label{M3-3}
\arg w_k(x) \in C(F_k \setminus e_k), \qquad \cop(e_k) = 0
 \qquad k =1, 2, \dots, s;
\end{equation}
It turns out that generalizations of two parts of the theorem \ref{M-1} to this settings have different status. More exactly, we have the following\\

\begin{thm} \label{M-5}
 Under assumptions \eqref{M3-2} and \eqref{M3-3} on functions $f_k = w_k$ in \eqref{M3-1} assertions \eqref{M-11} (B) and \eqref{M-12} of theorem \ref{M-1} remain valid for first kind Hermite--Pad\'e polynomials associated with $\vec f$.\\
\end{thm}

Theorem \ref{M-5} is not a direct corollary of theorem \ref{thm1} but can be proved following step by step the proof of theorem \ref{thm1} (see section 3 below) and even with significant simplifications.
Situation with 
the second kind Hermite--Pad\'e polynomials is different; in this case we only can state a conjecture.\\

{\bf Conjecture.}
{\it Under assumptions \eqref{M3-2} and \eqref{M3-3} on functions $f_k$ in \eqref{M3-1} assertions \eqref{M-11} (A) of theorem \ref{M-1} remain valid for the second kind Hermite--Pad\'e polynomials associated with $\vec f$.}\\

Conjecture looks very natural and it is confirmed by all known results.
But at the moment it is not clear how it may be proved in its
generality. We also can not prove in full generality the theorem
similar to theorem \ref{thm1} for second kind Hermite Pad\'e
polynomials. Some version of such theorem may be proved but more
restrictive assumptions on separation of singularities of component
functions will be needed. To make one step towards explanation of the
situation, consider a plan of the proof.

Suppose we want to proof theorem \ref{M-5} (and/or the conjecture)  using the same method of reduction of vector case to weighted scalar one. Then we have to begin with a complex version of  the theorem \ref{M-3}. Such a theorem is known.
\begin{thm}\label{M-6}
Under all the assumptions on $F, \  \Phi_n, \ Q_n$ in theorem \ref{M-3} (see \eqref{Field} and \eqref{Ort}) and the condition \eqref{M3-2} - \eqref{M3-3} on $w$ we have
 \begin{equation*} 
\frac1n\  \mc{X}\(Q_n\)=\frac{1}{n}\sum_{Q_n(\zeta)=0}
\delta(\zeta)\ \  {\overset{*}{\to}}\ \  \lambda
\end{equation*}
where $\lambda = \lambda_\varphi$ is the equilibrium measure on $F$ in the external field $\varphi$.
\end{thm}

Again, theorem \ref{M-6} is not a direct corollary (nor a particular
case) of more general theorem \ref{thm3-1} which we discuss in the next
section 3 but can be proved following same path. We outline briefly
some elements of this path which would, in particular, make an
introduction to the settings and proof of theorem \ref{thm3-1}.

The proof begins exactly like the one of theorem \ref{M-3}: we select a weakly convergent subsequence $\quad \frac1n\  \mc{X}\(Q_n\)  \to  \mu \quad$ where $\Lambda= \{n_k\}_{k=1}^\infty \subset \N. $   Now it is enough to show that $\mu = \lambda$.

Assuming the contrary $\mu \ne \lambda$ we will construct a sequence of
polynomials $P_n$ which leads to a contradiction with orthogonality
conditions \eqref{Ort}. In the proof of theorem \ref{M-3} we have
defined required polynomials $P_n$ by a simple explicit formula
\eqref{P_n}. Then we determined asymptotics of the integral involving
these polynomials
 \begin{equation*} 
 I_n =  \int _{F} Q_n(x) P_n(x)\Phi_n(x) w(x) dx
\end{equation*}
and derived from this asymptotics that $I_n \ne 0$ for large enough $n$ in contradiction with \eqref{Ort}.

The difference with the real case comes at this point. Assuming $\mu
\ne \lambda$ we have to construct polynomials  $P_n$ in such a way that
a lower bound for $|I_n|$ may be obtained. For complex weights $w$ it
will not be possible to use for this purpose polynomials $P_n$ in
\eqref{P_n} (such a low bound will not be valid for them). The method
of constructing $P_n$ for complex $w$ essentially modified. Loosely
speaking, we will construct $P_n$ in such a way that the product
$|Q_n(x) P_n(x)|\Phi_n(x)$ have a single sharp local maximum of at a
point $x_0$ where $w(x_0) >0$ and $\arg w$ is continuous. This idea (it
goes back to H.~Stahl) is the central point of the method. Formal
implementation of this plan is rather cumbersome; we postpone
discussion of the details until next section.


Important detail in the hypothesis of the theorem  \ref{M-5}  is that the assumption \eqref{Field} on the variable part $\Phi_n$ of the orthogonality weight  remain the same as in theorem \ref{M-3};  we assume that $\Phi_n(x) > 0$ (while assumptions on $w$ are more general).  Assertion of the theorem will not be valid for complex-valued $\Phi_n(x)$ which would lead a nonsymmetric $\varphi$ in \eqref{Field}; union of intervals $F \subset \R$ will not have the $S$-property in such an external. f

This explains also the reason why the conjecture above on the
asymptotics of second kind Hermite--Pad\'e polynomials is difficult to
prove using methods which were effective in the proof of theorem
\ref{M-3}. In the process of reduction of vector-orthogonal polynomials
to weighted scalar ones we introduce weights $\Phi_{n,k}(x) = \prod_{j
\ne k}P_{n,j}(x.)$ For for complex valued weights we loose any {\it a
priori} information on the location of zeros of polynomials $P_n =
\prod_{j =1}^s P_{n,j}(x)$ and their possible limit distributions.
Therefore, we can not assert that $F \subset \R$ has $S$-property in
the field $\varphi$ defined in \eqref{Field}.

The first kind Hermite Pad\'e polynomials turns to be technically simpler because of the specific structure of their orthogonality conditions \eqref{M6}. Asymptotics of the sum of integrals is reduced to the asymptotics of one of them. In each integral $\int_{F_k}$ the other polynomials $q_{n,j}$ with $j \ne k$ are not involved like they are involved in case of second kind polynomials. This would allow us to use in case complex weights the same arguments we have used in the real one.

Situation is similar for the case $\vec f \in \mc A$ and we will be able to prove a theorem on zero distribution for a general Angelesco case for first kind polynomials.




\section{Proof of theorem \ref{thm1}}

Now we turn to the problem of the limit zero distribution for complex (nonhermittian) orthogonal polynomials with analytic weights in complex plane.
One of the most general results in this direction
is the theorem 3 in \cite{GoRa87} (so-called $\GRS$-theorem), In this section we present a somewhat modified version of the theorem which is, then, used to the study of Hermite--Pad\'e polynomials. In the next section we discuss the original version of the theorem.

\subsection{The $\GRS$-theorem.}

The theorem in the original form, characterizes the limit zero distribution as $n\to \infty$ for a sequence of polynomials $Q_n(z) \in \P_n$ defined by the following complex orthogonality relations
\begin{equation}
\label{3-1}
\oint_\Gamma Q_n(z)z^k \Phi_n (z) f(z) dz=0,\quad k=0,1,\dots,n-1;
\end{equation}
with a variable (depending on $n$) weight function $w_n(z) =  \Phi_n
(z) f(z)$ where integration in $\oint_\Gamma$ goes over a cycle
surrounding compact set $\Gamma.$ Formal  definition of $\oint_\Gamma $
will be presented after we introduce conditions on $\Gamma$ and $f$
below.

Hypotheses of the theorem include  several conditions on compact set $\Gamma$, sequence of the functions $ \Phi_n (z) $ analytic on $\Gamma$ and a function $f$ analytic in a domain surrounding $\Gamma$. These conditions remain unchanged in the new version of the theorem. \\

{\bf Condition 1.}  The first condition is convergence of (analytic) functions $\Phi_n (z)$
Let $\Phi_n \in H(\mc O)$ where $\mc O$ is an open set contained in the disc $\{z:\  |z| < 1/2\}$ and
for the sequence of functions $ \Phi_n (z)$ holomorphic in $\mc O$ we have
uniformly on compact sets in $\mc O$ we have convergence
\begin{equation}
\label{3-2}
\frac 1{2n}\ \log \frac 1{|\Phi_n(z)|}\ \to \ \   \varphi(z)
\end{equation}
to a real function $\varphi(z)$ harmonic in each connected component of $\mc O.$

We note that the condition $\mc O \subset\{z:\  |z| < 1/2\}$ is technical. It can be easily replaced by  $\mc O \subset\{z:\  |z| < 1\}$ but if it is entirely omitted we should discuss more carefully $n$-th root asymptotics of spherical normalization of polynomials in terms of spherically normalized potentials.\\

{\bf Condition 2.} Compact $\Gamma $ in \eqref{3-1} belongs to  $ \mc O$ and has the $S$-property in the external field $\varphi$ from \eqref {3-2} (which is explained in the next definition \ref{def1}).
Moreover, the complement to the support of equilibrium measure $\lambda=\lambda_{\varphi, \Gamma}$ is connected.

We note that in the case of multipoint Pad\'e
approximants the complement
to the extremal compact set is not necessary a domain, i.e., it might
consists of several domains; see~\cite{BuMaSu12},~\cite{Bus13},~\cite{Bus15}.
In the case of Hermite--Pad\'e polynomials
for the general complex Nikishin systems the situation is very similar
to the multipoint Pad\'e approximants;
see~\cite{MaRaSu12},~\cite{RaSu13},~\cite{Sue15}.


\begin{definition}\label{def1}
Let $\mc O$ be an open set and a real function $\varphi(z)$ be harmonic
in each connected component of $\mc O.$ We say that a compact set
$\Gamma \in \mc F$ has an $S$-property relative to external field
$\varphi$ if there exist a set $e \subset \Gamma$ of zero capacity such
that for any $\zeta\in \Gamma\setminus e$ there exist a neighborhood
$D=D(\zeta)$ of $\zeta$ such that $\supp(\lambda)\cap D$ is an analytic
arc and, furthermore, we have

\begin{equation}
\label{3-2n}
\frac{\partial}{\partial n_1}\(V^\lambda+\varphi\)(\zeta)
=\frac{\partial}{\partial n_2}\(V^\lambda+\varphi\)(\zeta),\quad \zeta\in \Gamma^0 = \supp\lambda \setminus e
\end{equation}
where $\lambda=\lambda_{\varphi, \Gamma}$ is the equilibrium measure for $\Gamma$ in $\varphi$ and $n_1, n_2$ are two oppositely directed normals to $\Gamma$ at $\zeta\in \Gamma$ (we can actually admit that  $\varphi$ has a small singular set included in $e$).
\end{definition}

We will call set $\Gamma^0 = \supp(\lambda) \setminus e$ the regular part of $\Gamma^1 = \supp(\lambda) $. Clear that the regular part is finite or countable union of disjoint open analytic arcs.  \\

{\bf Condition 3} on the function $f$ will be stated as $f \in \mc
H_0(\mc O \setminus  \Gamma)$ where by $\mc H_0(\mc O \setminus
\Gamma)$ we denote the class of all functions $f \in H(\mc O\setminus
\Gamma)$ which also have the following property. There exists a set
$e_0 \subset \Gamma$ off zero capacity such that for any arc $\gamma
\subset \Gamma \setminus (e \cup e_0)$ function $f$ has continuous
extension to $\gamma$ from both sides of the arc and the difference  of
boundary values of $f$ does not have zero on $\gamma$.

Thus, $\Gamma$ is the set of singularities of $f$ and by $\oint_\Gamma$ in
\eqref{3-1} we understand integration over a union $C = \cup C_j$ of
mutually exterior piecewise smooth Jordan contours $C_j$  containing all
the singularities of $f$ in the union of their interiors $ \Gamma \subset
\bigcup_j \myint C_j $.

Now we state the theorem 3 from \cite{GoRa87} ($\GRS$-theorem).
\begin{thm}\label{thm3-1} Let polynomials $Q_n(z) \in \P_n$ be defined by orthogonality relations \eqref{3-1} and conditions 1, 2 and 3 respectively on $ \Phi_n (z), $ $\Gamma$ and $f$ are satisfied.
Then the following assertions hold

$(i)$
\begin{equation*} 
\frac1n\,\mc{X}\(Q_n\)\overset{*}{\ \to\ }\lambda, \qquad \text{as} \quad
n \to \infty
\end{equation*}
where $\lambda= \lambda_{\varphi, \Gamma}$.\\

$(ii)$ If $Q_n$ are spherically normalized then we have convergence in capacity
\begin{equation*} 
\left|\oint_\Gamma \frac {Q^2_n(z)}{\zeta - z} \Phi_n (z) f(z) dz\right|^{1/n} \ \to \ e^{ - 2w_\varphi}
\end{equation*}
where $w_\varphi$ is the equilibrium constant  \eqref{M7} for $\Gamma$
and $\varphi$ for spherically normalized total potential of $\lambda$.
\end{thm}

\subsection{Outline of the proof of theorem 8. Basic lemmas.}

We are interested here in the proof of part $(i)$. It begins with the
selection of a weakly convergent subsequence $\quad \frac1n\
\mc{X}\(Q_n\)  \to  \mu \quad$ where $n \to \infty,$
 $n \in \Lambda= \{n_k\}_{k=1}^\infty \subset \N. $
Now it is enough to prove that $\mu = \lambda$.  Assume the contrary $\mu
\ne \lambda$. Under this assumption we will construct a sequence of
polynomials $P_n$ with $\deg P_n <n$ such that
$$
\biggl|\oint_\Gamma Q_n(z)P_n(z) \Phi_n (z) f(z) dz\biggr|^{1/n} \to c > 0
$$
in contradiction with orthogonality
relations.

The same approach has been used in the proof of theorem \ref{M-3} above.
The construction of the sequence $P_n$ now is different. The procedure is
rather long and made in several steps.

First, starting with a given measure $\mu$ (representing a hypothetical zero distribution of orthogonal polynomials) we construct a measure $\sigma$ which will later play role of the limit zero distribution for the sequence $P_n.$ Properties of this measure $\sigma$ are listed in the
lemma 3 below which asserts existence for a measure with required properties.

To state the lemma we need to mention one important property of
compact sets with $S$-property in a harmonic field -- existence of a
reflection function $z \mapsto z^*$ in a neighborhood $\mc U(z_0)$ of any
regular points $z_0 \in \Gamma^0$. Actually, to assert existence of
such function we need only local analyticity of $ \Gamma^0.$ It follows
from analyticity of $\Gamma$ at $z_0$ that there a neighborhood  $\mc
U(z_0)$ of this point may be selected such that there is a conformal
mapping $\psi: \mc U(z_0) \to D_1 = \{z: |z| <1\}$ with $\psi(z_0) = 0$
and $\psi(\Gamma_0 \cup \mc U(z_0)) = (-1,1).$ Then for $z \in \mc
U(z_0)$ we define $z^* = \psi^{-1}(\overline {\psi(z)})$.

Thus, $z \mapsto z^*$ is anticonformal mapping similar to complex
conjugation. In particular, $z^* = z$ for points on $\Gamma_0$. By $\mc
U(z_0)$ we will denote $*$-symmetric neighborhoods of points $z_0\in
\Gamma_0.$ For a set $E \subset \mc U$ we denote $E^* = \{z^*: \ z \in
E\}.$ For any measure $\nu$ on $\mc U$ we denote by $\nu^*$ the reflected
measure defined by $\nu^*(E) = \nu(E^*), \quad E \subset \mc U.$ 
\cite{GoRa87} for further comments).

\begin{lemma}\label{3}
Let $\varphi$ be a harmonic function in an open set $\mc O\subset \CC$, a compact set $\Gamma \subset \mc O$ has $S$-property in the field $\varphi$ and the complement to the support of $ \lambda = \lambda_{\varphi, \Gamma}$ is connected (that is, condition 2 for $\Gamma$ is satisfied).

Suppose $\mu$ is a measure in plane with $|\mu| \leq 1$ and $\mu \ne
\lambda.$ Then there exists $\ol r = \ol r (\mu) \in (0,1)$, such that
for any $r \leq \ol r$, any positive measure $\eta$ with $|\eta| \leq
r$ with the support in $\CC \setminus \ol {\mc O}$, any $\varepsilon \in
(0, 1-r)$ and any compact set $e \subset \Gamma$ of capacity zero there
is a point $z_0 \in \Gamma^0\setminus e$, its $*$-symmetric
neighborhood $\mc U = \mc U(z_0) \subset D_\varepsilon(z_0)$ and a
positive measure $\sigma = \sigma(\mu, \eta)$ with the following
properties

(i) \quad $|\sigma| = 1-r + |\eta| + \varepsilon\quad $ 

(ii) \quad $\sigma|_{\mc U} = (\mu|_{\mc U})^*$

(iii) \quad $U^\sigma(z) = +\infty \quad $ for $\quad z \in e$.

(iv) \quad $(U^{\mu + \sigma} + 2\varphi)(z^*) = (U^{\mu + \sigma} +2 \varphi)(z)$
\quad  for \quad $z \in \mc U$

(v) \quad $(U^{\mu + \sigma} + 2\varphi)(z_0)  < (U^{\mu + \sigma} + 2\varphi)(z)$ \quad for \quad $z \in \Gamma \setminus \{z_0\}.$

(vi) \quad 
total variation of the signed measure $\sigma - (r\lambda+ \eta)$ is at most $\varepsilon$.
\end{lemma}
\begin{proof}
Hypotheses and assertions (i) - (v) of lemma 3 are in essence identical
to those in the lemma 9, p. 340 in \cite{GoRa87} and the construction
of the proof does not require significant modifications. Still,
technical changes have to be made. To place them in a proper context we
make a few comments on the structure of the proof.

We mention first a few differences in notations. The external field was
denoted $\psi$ in lemma 9 and $S$-compact set in this field was denoted by
$F$ (whose neighborhood where $\psi$ is defined was $U$); $\Gamma$ was
the support of equilibrium measure $\lambda.$ Notations for a potential
of a measure $\mu$ was by $V^\mu$. Here we use our notations ($\varphi$
for the field, $\Gamma$ for the $S$-compact set $\mc O$ for its
neighborhood and so on). There are few more minor differences which
should not cause problems.

The main modification in the lemma is the assertion (vi) meaning that
an arbitrary small enough positive measure $\eta$  may be included as
part of $\sigma.$ Measure  $\sigma$ is explicitly constructed in the
process of the proof of lemma 9 in \cite{GoRa87} in such a way that
only an arbitrary small part of the constructed measure $\sigma$ is not
known explicitly.
The new property (vi) of the measure $\sigma$ means that this measure may be constructed in such a way   that it has a form $\sigma = r\lambda +\eta+\nu$ where $\eta$ is an arbitrary positive measure with $|\eta| < 1 - r$ and $r \leq \ol r(\mu) <  1.$

The additional (signed) measure $\nu$ may be made arbitrary small in variation. We note that $\sigma$ is a positive measure, but $\nu$ is generally not. Measure $\nu$ is essentially constructive too, but it plays a technical role and its form is not really important. In lemma \ref{5} below we will show that it is possible to get rid of this part passing to a limit.

To explain the way the measure $\eta$ is incorporated in $\sigma$ we reproduce the beginning of the proof of lemma 9. The parameters in the proof of lemma 9 appear in the following order. First, for a given measure $\mu$ and a function $W$ is introduced by the formula
\begin{equation*} 
W(z)  = W(z; t, \eta) = U^{\mu +t\eta}(z) +(1+2t)\varphi(z) - (1-2t) (U^\lambda + \varphi)(z)
\end{equation*}
where $\eta$ is an arbitrary unit measure  and $t \in [0,1/2]$.  Then two associated functions are defined
\begin{equation*} 
w_0(t, \eta)   = \inf_{\Gamma^1} W(z; t, \eta), \qquad  w_1(t, \eta)   = \inf_{\partial \mc O} W(z; t, \eta)
\end{equation*}
Since $ w_0(0, \eta) < w_1(0, \eta)$ (note that this quantities do not
depend on $\eta$) it is concluded that there is $t \in (0,1/2)$ such that
$ w_0(t, \eta) < w_1(t , \eta)$  for any unit measure $\eta$ with the
support in $\ol {\mc O}$ ($t$ may be made smaller). Value of $t$ is fixed
from this point on.

Then the procedure of selection of $\eta =(1/2) (\eta_1 +\eta_2)$  is described as follows. Unit measure $\eta_1$ is selected on $e$ such that  $U^{\nu_1} = +\infty$ on $e$. Then a point $z_0 \in \Gamma^0$ is selected where $U^{\mu+ (t/2)\eta_1} + \varphi $ assumes the minimum value over $\Gamma$. Since $U^{\mu+ (t/2)\eta_1} + \varphi  = +\infty$ on $e$, the  point $z_0$ is interior point of one of the open analytic arcs of $\Gamma^0.$

Then a unit measure $\eta_2$ is selected whose potential $U^{\nu_2}(z)$ has strict minimum over $\Gamma$ at the point $z_0$ 

Now, as part of the proof of lemma \ref{3} the construction of $\eta$
has to be modified as follows. First, we define $\ol r  = 1- t$ with
the above choice of $t$.   Then we select arbitrary unit measure with a
compact set support in  $\CC \setminus \ol {\mc O}$. From this point on we
follow the procedure of proof of lemma 9. At this point it would be
convenient to modify notations. The parameter $t$ is already renamed to
$1 - r$ and this number is fixed. The total mass of the other
components of $\sigma = r\lambda +\eta+\nu$ which we denote by $\nu$
just need to be small. We fix a positive  $\varepsilon < (1 - r)/4$.

Totally we will have 4 components in measure $\nu$. Two of them $(t/2)
(\eta_1 +\eta_2)$ selected above have to be renormalized so that total
masses are $\leq \varepsilon$. The third component (denoted $\nu$ in
lemma 9) is defined as $(\mu|_{\mc U})^*$ (see (ii)). Choice of $\mc U$
allow us to make it small. Finally the last component is also supported
in an arbitrary small neighborhood of $z_0$ and also may be made
arbitrary small.
\end{proof}

Next, we state as lemma \ref{4} another result from \cite{GoRa87} (lemma 8, p. 337). We use a few different notations here. In particular, our $\mc O, \Phi_n (z), \varphi, G_n, g_n$ are denoted respectively as $\mc U, \Psi_n, \psi/2, P_n, q_n$ in  \cite{GoRa87}, lemma 8. Content is not changed and the proof in \cite{GoRa87} remains valid.

\begin{lemma}\label{4} Let $\mc O$ be a domain, $L\subset \mc O$ be a simple closed analytic arc and functions $ \Phi_n (z)\in H(\mc O) $ satisfy condition 1 (see \eqref{3-2n}).

Let furthermore $w(z)$  be a continuous function on $L$ without zeros and $G_n(z)$  be sequence of polynomials with
\begin{equation} \label{39}
\nu_n: = \frac1n\,\mc{X}\(G_n\)\overset{*}{\ \to\ }\nu \quad \text{as}  \quad n\to \infty, \quad
n \in \Lambda = \ \{n_k\}_{k =1}^\infty.
\end{equation}

Suppose the following conditions are satisfied in some $*$-symmetric neighborhood $\mc U$ of an interior point $z_0 \in L$

(i) \quad $\nu_n|_{\mc U} = (\nu_n|_{\mc U})^*$ and all zeros of $G_n$ on $L\cap \mc U$ are of even multiplicity,

(ii) \quad $U^{\mu + \sigma}(z^*) = U^{\mu + \sigma}(z)$\quad  for \quad $z \in \mc U$

(iii) \quad $U^{\mu + \sigma}(z_0)  < U^{\mu + \sigma}(z)$ \quad for \quad $z \in L \setminus \{z_0\}.$

Then there exist  a sequence of polynomials $g_n(z)$ with zeros in $\mc
U$ with $(\deg g_n)/n \to 0$ as $ n\to \infty$ such that

\begin{equation} \label{40}
 \left|\int_L g_n(z)\ G_n(z) \Phi_n (z) w(z)(z) dz\right|^{1/n } \
 \to \  \exp\{- ( U^{\nu} + 2\varphi) (z_0)\}, \quad
\end{equation}
(potential and polynomials are spherically normalized).
  \end{lemma}

The original proof of $\GRS$-theorem in \cite{GoRa87} was a combination of the lemmas 8 and 9 and one more rather simple lemma 7 asserting inequality \eqref{46}, \eqref{47} below.  Here it is convenient for our purposes to single out one more intermediate lemma combining lemmas \ref{3} and \ref{4}.

  \begin{lemma}\label{lem5} Let conditions 1, 2 and 3 respectively on $ \Phi_n (z), $ $\Gamma$ and $f$ are satisfied and a sequence of spherically normalized polynomials $Q_n(z) \in \P_n$ has limit zero distribution
 \begin{equation} \label{42}
\mu_n: = \frac1n\,\mc{X}\(Q_n\)\overset{*}{\ \to\ }\mu \quad \text{as}  \quad n\to \infty, \quad
n \in \Lambda = \ \{n_k\}_{k =1}^\infty
\end{equation}
along a subsequence  $\Lambda$ represented by a measure $\mu$ with
total mass $\mu(\CC) \leq 1$ and different from equilibrium measure
$\lambda = \lambda_{\varphi,\Gamma}$.

Let $\eta$ be a measure with a compact set support in $\CC \setminus \ol
{\mc O}$ and $|\eta| \leq \ol r (\mu).$ Let  $\sigma = \sigma(\mu,
\eta) = (1-r)\lambda + \eta +\nu$ be a measure associated with $\mu$
and $\eta$ whose existence is asserted in lemma 3 above, in particular
$|\nu| < 1-r$ Then there exists spherically normalized sequence of
polynomials $P_n$ satisfying conditions $\deg P_n < n$,
 \begin{equation} \label{43}
\mu_n : = \frac1n\,\mc{X}\(P_n\)\overset{*}{\ \to\ }\sigma \quad \text{as}  \quad n\to \infty, \quad
n \in \Lambda = \ \{n_k\}_{k =1}^\infty
\end{equation}
and
\begin{equation} \label{44}
 \left|\oint_\Gamma Q_n(z)\ P_n(z) \Phi_n (z) f(z) dz\right|^{1/n } \ \to \  e^{-m}
\end{equation}
where
\begin{equation*} 
m = m(\mu+\sigma) =  \min_{z \in \Gamma}( U^{\mu + \sigma} + 2 \varphi) (z) =
 ( U^{\mu + \sigma} + \varphi) (z_0)
\end{equation*}
(potential of $\mu + \sigma$ is spherically normalized).
  \end{lemma}
\begin{proof}
As we mentioned, lemma 5 is essentially a combination of lemmas \ref{3} and \ref{4}. Reduction to lemmas \ref{3} and \ref{4} require mainly some comments on connections between the two lemmas.

First, exceptional set $e$ which appear in hypotheses of lemma 3 will
be union the three sets, First one is $\Gamma^1 \setminus \Gamma^0$ end
points of analytic arcs of $\supp \lambda$. Here will be also included
points where equilibrium conditions \eqref{M7} are violated (if any).
Second one is $e(f)$ (singularities of $f$). On any arc in  $\Gamma^0
\setminus e(f)$ function $f$ has continuous boundary values $f^\pm (z)$
and the third part of exceptional set  consists of zeros of $w(z) =
(f^+-f^-)(z)$ on $\Gamma^0 \setminus e(f).$

Next, conditions of lemma \ref{3} are satisfied and the lemma asserts existence of a point
$z_0 \in \Gamma^0\setminus e$ and its $*$-symmetric neighborhood $\mc U = \mc U(z_0) \in \Gamma^0\setminus e$ and a measure $\sigma = \sigma(\mu)$ with properties (i) - (vi).

We define analytic arc $L = \Gamma^0 \cup \mc U$ and  $w(z) = (f^+-f^-)(z)$ on $L.$
Then we construct a sequence of polynomials $\tilde P_n$ with  $\frac1n\,\mc{X}(\tilde P_n)\overset{*}{\ \to\ }\sigma.$ This can be done with a significant degree of freedom and the choice of $\tilde P_n$ may be made subject some extra conditions.

Namely, using (ii) of lemma 3 we can select zeros of $\tilde P_n$ in $\mc U$ so that they are $^*$-symmetric to zeros of $Q_n$ in $\mc U.$ In particular, zeros of $\tilde P_nQ_n$ on $L = \Gamma \cup \mc U$ are of even multiplicity. Thus, we infer that condition (i) of lemma \ref{lem5} for the measure $\nu = \mu +\sigma$. The other two conditions (ii) and (iii) on potential of $\nu$ in lemma \ref{4} are satisfied since they are identical to assertions (iv) and (v) of lemma 3.

We define for $G_n = \tilde P_n Q_n$. Now, conditions of lemma \ref{4} are satisfied and the lemma asserts existence of polynomial $g_n$ of degree $o(n)$ with \eqref{40}. We set $P_n = g_n \tilde P_n.$
For these polynomials the assertion \eqref{44} in lemma \ref{lem5} with $\oint_\Gamma$ in \eqref{39} replaced by $\int_L$ (and $w$ in place of $f$)  follow from lemma {4}.

Finally, to pass to $\int_\Gamma$ in \eqref{44} we use inequality
\begin{equation} \label{46}
\varlimsup_{n\in\Lambda} \left|\int_{\Gamma \setminus L} Q_n(z) P_n(z) \Phi_n (z) f(z) dz\right|^{1/n } \ \leq \  e^{-m'}
\end{equation}
where $m' = \min_{\Gamma \setminus L} ( U^{\mu + \sigma} + 2\varphi) > m$ and contour $\Gamma$ is modified so that (both copies of) $L$ belong to $\Gamma.$
This assertion is actually follows from quite general properties of convergence of potentials of weakly convergent measure (see e.g.  lemma 7 in \cite{GoRa87} and comments there).
\end{proof}

\subsection{A modification of lemma 5}
Lemma 5 may present, after a proper generalization, an independent interest.
We mention briefly a more general problem connected to the lemma.

Consider again setting of the lemma: given $ \Phi_n (z), $ $\Gamma$ and $f$ with conditions 1, 2 and 3 respectively and two sequences of spherically normalized polynomials $Q_n$ and $P_n$ satisfying conditions \eqref{39} and \eqref{40} respectively. Now, return to inequality \eqref{46} which we write in application to the whole set $\Gamma$
\begin{equation} \label{47}
\varlimsup_{n\in\Lambda} \left|\oint_\Gamma Q_n(z) P_n(z) \Phi_n (z) f(z) dz\right|^{1/n } \ \leq \  e^{-m}, \quad \text{where}\quad m = \min_{\Gamma}( U^{\mu + \sigma} +2\varphi)
\end{equation}
(see lemma 7 in \cite{GoRa87}). As mentioned above, this inequality is rather general, it follows from $\frac1n\,\mc{X}\(P_nQ_n\)\overset{*}{\ \to\ }\mu +\sigma$ and some mild assumptions on $\varphi$,  $f$, $\Gamma$  (no need for the $S$-property) and convergence in \eqref{3-2}.

Now, a natural questions to ask are the following. Suppose a sequence $Q_n$ of polynomials with \eqref{42} is given. What are conditions (in particular on measures $\mu, \sigma$) which would imply that there is a sequence $P_n$ with \eqref{43} such that equality in \eqref{47} holds? Another question is what are conditions under which to the limit exists in place of the upper limit.

Here we are not going to investigate these problem in general settings. For the purposes of application to Hermite--Pad\'e polynomials we need to prove one particular proposition which is a corollary of lemmas \ref{3} -\ref{lem5}.
Loosely speaking we will show that in setting of lemma 5 equality in
\eqref{47} holds for any $\mu \ne \lambda$ if $\sigma - \lambda$ is
small enough. We do not ask if $\lim$ exists, so that, we are ready to
pass to a subsequence of $\Lambda$.


\begin{lemma}\label{lem6}
Let $\mc O \subset D_{1/2}$ be an open set and for the sequence of functions $ \Phi_n \in H(\mc O)$, compact set $\Gamma \subset \mc O$ and function $f \in H_0(\mc O \setminus \Gamma)$ conditions 1, 2 and 3 are satisfied. Let $Q_n(z) $ be a sequence of polynomials satisfying \eqref{42} with a measure $\mu \ne \lambda$ with $|\mu| \leq 1.$

Then there exist $\bar r \in (0,1)$ such that for any $r \in [\bar r, 1]$ and any measure $\eta$ with a compact set support in $\CC \setminus \ol {\mc O}$ and $|\eta| = 1- r $ there exists a sequence of polynomials $P_n, \ n\in \Lambda$ such that $\quad\frac1n\,\mc{X}\(P_n\)\overset{*}{\ \to\ } r\lambda + \eta\quad$ as $\quad n\to \infty$, $ \quad\ n\in \Lambda, \quad$ such that $\deg P_n < n,$ all zeros of $P_n$ belong to $\Gamma \cup \mc O$  and
\begin{equation} \label{48}
 \left|\oint_\Gamma Q_n(z) P_n(z) \Phi_n (z) f(z)\ dz\right|^{1/n } \ \to \
\exp \{ - \min_{\Gamma} (U^{\mu + \sigma} + 2\varphi )\}
\end{equation}
is valid with $\sigma = r\lambda + \eta.$ (potentials and polynomials are spherically normalized).

\end{lemma}
\begin{proof}
We fix exceptional set $e$ as in the proof of lemma \ref{lem5}.  According to lemma \ref{3} there exists $\ol r = \ol r(\mu) \in (0,1)$ and a measure $\eta$ with $|\eta| \leq 1 - r$ such that for any $\varepsilon \in (0, (1-r)/4)$ there is a point $z_0 = z_0(\varepsilon)  \in \Gamma^0\setminus e$, its $*$-symmetric neighborhood $\mc U = \mc U(z_0) $ and a positive measure $\sigma = \sigma(\mu, \varepsilon)$ with properties (i) - (vi). We will pass to the limit as $\varepsilon \to 0$

We fix $r \in [\ol r, 1)$  and measure $\eta$ with $|\eta| < 1- r$ .
Consider the sequence $\varepsilon = 1/N$ with $N \in \N$. Denote associated sequence of measures $\sigma$ by $\sigma_N = \sigma(\mu, 1/N) = r \lambda + \eta + \nu_N$.
By lemma \ref{5} for any $\sigma_N$ we have a sequence of polynomials  $P_{n, N}$ with $\quad\frac1n\,\mc{X}\(P_{n,N}\)\overset{*}{\ \to\ }\sigma_N \quad$ and
\begin{equation}
\label{49}
M_{n,N} = \left|\oint_\Gamma (Q_n\ P_{n, N}\ \Phi_n \ f)(z)\ dz\right|^{1/n } \ \to \  M_N =
\exp \{ - \min_{\Gamma} (U^{\mu + \sigma_N} + 2\varphi )\}
\end{equation}
as $ n\to \infty$, $ n \in \Lambda = \ \{n_k\}_{k =1}^\infty$. Note that the sequence $ \Lambda $ do not depend on $N$ (it is the sequence of convergence of counting measures for $Q_n$; see \ref {42}).

By (vi) in lemma \ref{3} we have convergence $\sigma_N \to r\sigma$ (in total variation and, therefore, weak$^*$. It follows $M_N \to M = \exp \{ - \min_{\Gamma} (U^{\mu + r\lambda} + 2\varphi )\}.$
Finally, using standard procedure we can select a diagonal subsequence in the table $M_{n, N}$ which converges to $M$. This gives a subsequence  $ \Lambda_1 \subset  \Lambda$ and sequence
$N_n$,
$n \in \Lambda_1$ such that $N_n \to\infty$ and $M_{n, {m_n}} \to M$ as $n \to\infty,$ $\ n \in \Lambda_1.$ Since $\sigma_N \to r\sigma$ in variations, we also have $\sigma_{N_n} \to r\lambda$.

Finally, selection of zeros of $P_n$ can be made so that  $\quad\frac1n\,\mc{X}\(P_n\)\overset{*}{\ \to\ } r\lambda \quad$ as $\quad n\to \infty, \ n\in \Lambda_1 $, and all zeros of $P_n$ belong to $\Gamma \cup \mc O.$ Thus, proof of the lemma is completed.
\end{proof}

Next we prove a lemma concerning the case $\mu = \lambda$ which was not
contained in lemma \ref{lem6}. The case is, indeed, exceptional. First,
the case is exceptional in lemma \ref{5} too, so we can not make a
direct reference to this lemma. Second we can not expect to construct
polynomials $P_n$ with \eqref{48} and $\deg P_n < n$ like in lemma
\ref{lem6}. Such polynomials should satisfy $\deg P_n \geq n$ (otherwise
orthogonal polynomials will provide us with a counter example to the
assertion of the lemma).

This suggests that we will need to use sequence of measures $\sigma_N \to \lambda$ approximating $\lambda$ from above in sense that $|\sigma_N| >1.$ In the proof of the next lemma \ref{7} we implement such approximation. Actually, for the future references we need the method of the proof of the lemma rather then its assertion.

\begin{lemma} \label{7}
Let $\mc O \subset D_{1/2}$ be an open set and for the sequence of functions $ \Phi_n \in H(\mc O)$, compact set $\Gamma \subset \mc O$ and function $f \in H_0(\mc O \setminus \Gamma)$ conditions 1, 2 and 3 are satisfied. Let $Q_n(z) $ be a sequence of polynomials satisfying \eqref{42} with $|\mu| \leq 1.$
Then there exists a sequence of polynomials $P_n, \ n\in \Lambda_1 \subset  \Lambda$ such that $\quad\frac1n\,\mc{X}\(P_n\)\overset{*}{\ \to\ }\lambda \quad$ as
$\quad n\to \infty, \ n\in \Lambda_1 $ and  \eqref{48} holds with $\sigma = \lambda.$
\end{lemma}
\begin{proof}
To make it possible to apply formally  lemma \ref{3} to a sequence of polynomials $Q_n$ with $\quad\frac1n\,\mc{X}\(Q_n\)\overset{*}{\ \to\ }\lambda \quad$ we argue as follows.
Fix $R >1$ and consider a new sequence of positive integers
\begin{equation*}
n' = [Rn] \in \Lambda', \qquad \text{where}\quad  n \in \Lambda
\end{equation*}
and where $[x]$ stands for the integral part of $x$. Next, we make
substitution in index (variable) $n$ of the polynomial (index is interpreted here as an independent variable of the function $n \to Q_n$).
We will consider polynomial $Q_n$ as element of $\P_{n'}$, that is, sequence $Q_n, \ n \in \Lambda$ is now interpreted as sequence $Q_{n'}, \ n \in \Lambda'.  $
We have
\begin{equation*}
\quad\frac1{n'}\,\mc{X}\(Q_{n'}\)\overset{*}{\ \to\ } r\lambda,
\qquad \text{where} \quad r = \frac 1R <1
\end{equation*}
So, the limit measure is now different form $\lambda$

Next, we need to make the same substitution in index $n$ of the variable weights $\Phi_n(z).$ This would result in the change of the function $\varphi$ representing asymptotics of $\Phi_n(z).$  In place of  \eqref{3-2} we have to write
\begin{equation}
\label{50}
\frac 1{2n'}\ \log \frac 1{|\Phi_n(z)|}\ \to \ \ \psi(z) =  R \varphi(z)
\end{equation}
Finally, compact set $\Gamma$ has $S$-property in the field $\varphi$ but
not in the field $\psi$ and to use lemma 5 we have to modify $\Gamma$.

More exactly, it would be enough to prove that for small enough $R-1 >0$
there is a compact set $\Gamma' = \Gamma(R)$ with $S$ property in the field
$\psi $ homotopic to $\Gamma$, so that $\oint_\Gamma$ in \eqref{48} may
be replaced with $\oint_\Gamma'$ and such that as $R \to 1$ we have
convergence of compact sets $ \Gamma(R) \to \Gamma$  in Hausdorff metric
and also weak convergence $\lambda' = \lambda(R) \to \lambda$ of their
equilibrium measures in associated fields $\psi = R\varphi.$ We will
prove such a lemma in section 4.

Now, for a fixed $R$ and $\Gamma' = \Gamma(R)$ we obtain from lemma \ref{5} that
there exists a sequence of polynomials $P_{n'}, \ n'\in \Lambda' \subset  \Lambda$ such that $\quad\frac1{n'}\,\mc{X}\(P_{n'}\)\overset{*}{\ \to\ }\lambda' \quad$ as $\quad n'\to \infty, \ n\in \Lambda' $, 
such that
 \begin{equation} \label{51}
 \left|\oint_{\Gamma'} Q_n(z)\ P_n(z) \Phi_n (z) f(z) dz\right|^{1/n } \ \to \  e^{- m}
\end{equation}
where $m = m(R) = R \min_{z \in \Gamma'}( (1+r) U^{\lambda'} + 2 \psi) (z). $ The constant $m(R)$ in the exponent has been obtained as follows. First we write this relation with $n'$ in place of $n$ and use expressions from lemma \ref{5}. Next, we make
substitution $n' = nR$ what changes the exponent $1/n'$ to $r (1/n)$. Finally, we raise both sides of the equality to the power $R$ which brings multiplier $R$ in the exponent to  the right.

 To conclude the proof we will consider sequence $R_N = 1+1/N \to 1$ and then use a diagonal process to find a desired sequence of polynomials. This part is similar to what we did in the proof of part (ii). By a proper selection of a sequence $N_n$ we can obtain a sequence of polynomials $P_n = P_{n, N_n}$ such that $\frac 1n \mc X(P_n) \to \lambda$ along some subsequence of $\Lambda$. Since $m(R) \to m_0 =
\min_{z \in \Gamma}( 2 U^{\lambda} + 2 \varphi) (z)$ as $R \to 1$ the
left hand side of \eqref{51} will converge to $e^{-m_0}$.
\end{proof}

\subsection{Proof of theorem 1.}
 Let $q_{n,k}(z) = c_{n,k}Q_{n,k}(z)$ be 
 polynomials defined in \eqref{2} and coefficients $c_{n,k}$ are determined by the condition that polynomials $Q_{n,k} \in P_n$ are spherically normalized.

Let $\vec \Gamma$ be the extremal compact set associated with $\vec f$ and
$\vec \lambda = \lambda_{\vec \Gamma}$ be its vector equilibrium measure
with unit components $\lambda_k;$ we denote $\lambda =\sum^s_{i=1}
\lambda_i$ and furthermore, $w_k = \min_{z \in \Gamma_k} W_k(z)$ where $W_k
(z) = U^{\lambda_k + \lambda}(z)$ (see \eqref{12}, \eqref{13}).  
components of the equilibrium vector potential

The beginning of the proof of theorem \ref{thm1} is similar to that in case of Markov functions in section 2. First we select a sequence $\Lambda \subset \N$ such that as $n \to \infty, \  $ $n \in \Lambda$ we have
\begin{equation} \label{52}
\frac 1n \mc X(Q_{n,k}) \to \mu_k, \qquad  |c_{n,k}|^{1/n}  \to e^{ - u_k}, \qquad k =1,2, \dots, s
\end{equation}
where $u_k \in[-\infty, +\infty].$ The remainder $R_n$ in \eqref{2} may
be normalized in such a way that $u_k \geq 0$ and $\min u_k = 0$. Then
we have $u_k \in[0, +\infty]$ and infinite values of $u_k$ are {\it a
priory} not excluded.

Now we have to prove that $\mu_k = \lambda_k$ and the numbers $\tilde m_k = u_k + w_k$ are all equal for $k =1, \dots, s$. 
We will do it in two steps. First, we will define
\begin{equation} \label{53}
m_k = u_k + \min_{z \in \Gamma_k}U^{\mu_k + \lambda}(z) =
 u_k + w_k +\min_{z \in \Gamma_k}U^{\mu_k - \lambda_k}(z),
\qquad k =1,2, \dots, s
\end{equation}
and prove, that $m_1 = m_2 = \dots = m_s$. Then we will prove that $\mu_k = \lambda_k$ which would also imply that $m_k = \tilde m_k$, thus,  completing the process.

In both steps we proceed from the contrary: assuming that the desired assertion is not valid we come to a contradiction with orthogonality relations \eqref{11} which may be equivalently written as follows
\begin{equation} \label{54}
\sum_{k=1}^s I_{n,k} = 0, \quad\text{where}\quad  I_{n,k} = I_{n,k}(G_n)  =
\oint_{\Gamma_k} q_{n,k}(z) G_n(z) f_k(z) dz
\end{equation}
and $G_n \in\PP_{ns}.$ For the sake of convenience of the reader we will first present detail of the proof for the case $s =2$.

\subsubsection{Proof of theorem 1 for the case $s =2$.}
In this section we have $k =1, 2$. To prove that $m_1 = m_2$  for constants $m_k$ in \eqref{54}
we assume the contrary $m_1 \ne m_2.$  Without loss of generality we can assume $m_1 < m_2$ (renumerating things in case of need).

We select a spherically normalized sequence of polynomials $\Phi_n$ with $\frac 1n \mc X(\Phi_n) \to \lambda_2$ and zeros on $\Gamma_2$.  For this sequence we have uniformly in a neighborhood of $\Gamma_1$
\begin{equation}
\label{55}
\frac 1{2n}\ \log \frac 1{|\Phi_n(z)|}\ \to  \   \varphi(z) = \frac 12 U^{\lambda_2} (z)
\end{equation}
Now conditions of lemma \ref{lem6} are satisfied with $\Gamma = \Gamma_1$,
$\ Q_n = c_{n,k}^{-1}q_{n,k}$ from \eqref{52}, $\mu = \mu_1$ from
\eqref{53} and $\sigma = \lambda_1$. It follows from the lemma  that
there is a (sub)sequence of polynomials $P_n$ with $n \in \Lambda_1
\subset \Lambda$ and $\frac 1n \mc X(P_n) \to \lambda_1$ such that
\begin{equation} \label{56}
 \left|\oint_{\Gamma_1} Q_n(z) P_n(z) \Phi_n (z) f(z)\ dz\right|^{1/n }  \to
\exp\{ - \min_{\Gamma} (U^{\mu + \lambda_1} + 2\varphi)  \} =
\exp\{ - \min_{\Gamma} (U^{\mu+ \lambda})\}
\end{equation}
where $\lambda = \lambda_1 +\lambda_2.$  We multiply this relation by $|c_{n,k}|^{1/n}$ and write it in terms of integrals
\begin{equation} \label{57}
I_{n,k} = I_{n,k}(G_n)  = \oint_{\Gamma_k} q_{n,k}(z) G_n(z) f_k(z) dz =
c_{n,k}\oint_{\Gamma_k} Q_{n,k}(z) G_n(z) f_k(z) dz
\end{equation}
with $G_n = P_n\Phi_n$ (compare to \eqref{54} above). Taking into account \eqref{53} we come to
\begin{equation*} \label{}
 \left|I_{n,1}\right|^{1/n } \ \to \  e^{- m_1} \quad \text{where} \quad m_1 =
 u_1 + \min_{\Gamma_1 } U^{\mu_1+ \lambda}
\end{equation*}
($m_1$ is the same constant as in \eqref{53} for $k=1$). For the second integral of $q_{n.2}$ over $\Gamma_2$ with the same $G_n = P_n\Phi_n$ we need only the upper bound
\begin{equation*} \label{}
\ol {\lim}  \left|I_{n,2}\right|^{1/n } \  \leq  \  e^{- m_2} \quad \text{where} \quad m_2 =
 u_2 + \min_{\Gamma_2 } U^{\mu_2+ \lambda}.
\end{equation*}
Since $m_1 < m_2$ the last two relations combined prove that $I_{n,1} +
I_{n,2} \ne 0$ for large enough $n \in \Lambda_1.$ This will contradict
orthogonality relations if $\deg G_n \leq 2n$. This inequality may
certainly be satisfied by our construction of polynomials $P_n, \Phi_n$
if $\mu_1 \ne \lambda_1$. In this case we come to a contradiction
showing that $m_1 = m_2.$

The exceptional case $\mu = \lambda_1$ require significant additional efforts (for a moment it is not clear how to avoid complications). The problem is that we can not use assertion of the lemma \ref{7} for the following reason. Suppose, like in case $\mu \ne \lambda_1$, we have selected a spherically normalized sequence of polynomials $\Phi_n$ with $\frac 1n \mc X(\Phi_n) \to \lambda_2$ and zeros on $\Gamma_2$.  For this sequence we have \eqref{55} uniformly in a neighborhood of $\Gamma_1$ and by lemma \ref{7} we can select sequence of polynomials $P_n, \ n\in \Lambda_1 \subset  \Lambda$ such that $\quad\frac1n\,\mc{X}\(P_n\)\overset{*}{\ \to\ }\lambda \quad$ as
$\quad n\to \infty, \ n\in \Lambda_1 $ and  \eqref{48} holds with $\sigma = \lambda.$ Then we define
$G_n = P_n\Phi_n$ and we can not use this polynomial to come to a contradiction with orthogonality conditions. Polynomial $P_n$ whose exact degree is out of control is selected after $\Phi_n$ and we can not prove that $\deg G_n \leq 2n.$

For this reason to prove that $m_1 = m_2$  true without any restriction on $\vec \mu$ we will use the proof of the lemma \ref{7}. More exactly, we will use the method of the proof of lemma \ref{7} combined with certain modification of the basic equilibrium problem.

We consider more general vector equilibrium problem which assign different total masses to components of the vector equilibrium measure.  The settings associated with an arbitrary vector $\vec t$ was outlined in the introduction. Here we need one-parametric family of equilibrium problems.
Let $t > 0$ be small enough; we introduce a vector $\vec t = (1+t, 1- t)$ whose components will represent total masses $t_k = \lambda_k(F_k)$ in the equilibrium problem associated with the class of vector measures
\begin{equation} \label{58}
{\vec {\mc M}}^{t}={\vec {\mc M}}^{\vec t}(\vec{F}) =
\left\{\vec{\mu}
=(\mu_1,\mu_2):\mu_j\in\mc {M}^{t_j}\(F_j\)\right\}, \quad t_1 = 1+ t, \ t_2 = 1-t
\end{equation}

All the definitions of parameters associated the equilibrium problems
\eqref{4} - \eqref{8} are modified in a clear way for this more general
equilibrium problem and become functions of $t$. Let $\vec \Gamma(t) =
\vec\Gamma(t, \vec f)$ be the extremal compact set  associated with max -
min energy problem (see \eqref{4} - \eqref{8}) in class ${\vec {\mc
M}}^{t}$ and ${\vec \lambda}^t  = (\lambda_1^t, \lambda_2^t).$
Actually, vector equilibrium measure and compact set depend on $t$
analytically.  For a moment we need only to know that the dependence
$\vec \Gamma(t)$ and  ${\vec \lambda}^t $ is continuous at $t=0.$ In
particular, if Angelesco condition is satisfied for the extremal vector
compact set $\Gamma = \Gamma(0)$ then there is $\delta >o$ such that
$\Gamma(\vec t)$ depends on $\vec t$ continuously for $| t | < \delta$
(Hausdorff metric is assumed in space vector compact sets; see definition
in sec. 4.1 for details).

It follows that constants $m_k = m_k(t)$ in \eqref{53} are continuous
as functions of $t$ (defined with the same vector measure $\mu$  from
\eqref{52}). Thus, we can select $t \in (0,1)$ such that inequality
$m_1(t) < m_2(t)$ is still valid.  We fix this value of  $t$.

Now, we can use the same approach as in the proof in case $\mu \ne \lambda_1$.
The situation here is quite similar to the proof of lemma \ref{7} and there is no need to repeat again all the details. In short, we are using the following difference between cases $t =0$ and $t>0$: the unit measure $\mu_1 = \lambda_1$ (which remains the same) can not be equal to the (new) equilibrium measure $\lambda_1$ since the normalization $|\lambda| = 1+t >1$ is different.
Thus, we, indeed, are back in case $\mu_1 \ne \lambda_1$ which was fairly treated above. By this the proof of equality $m_1 = m_2$ is completed. \\

It remains to prove that $\mu = \lambda$ if we know that $m_1 = m_2$. If it is not true then we may assume without loss of generality that $\mu_1 \ne  \lambda_1.$ From this assumption we will come to a contradiction.

We follow a procedure similar to the one used above to prove that $m_1 = m_2$. The first step is the same; we select a spherically normalized sequence of polynomials $\Phi_n$ with $\frac 1n \mc X(\Phi_n) \to \lambda_2$ and zeros on $\Gamma_2$.  For this sequence \eqref{55} is valid uniformly in a neighborhood of $\Gamma_1$

A change in the selection of $\sigma$ is made in the next step. We will apply lemma \ref{6} with $\Gamma = \Gamma_1$, $\ Q_n = c_{n,k}^{-1}q_{n,k}$ from \eqref{52}, $\mu = \mu_1$ from \eqref{53} and
\begin{equation*}
\sigma = (1-t)\lambda_1 + t\eta. 
\end{equation*}
where $\eta$ is the balayage of $\lambda_1$ onto $\Gamma_2$ and $ t >
0$ is small enough so that conditions of lemma \ref{6} are satisfied.
It follows from the lemma  that there is a (sub)sequence of polynomials
$P_n$ with $n \in \Lambda_1 \subset \Lambda$ and $\frac 1n \mc X(P_n)
\to \sigma$ such that
\begin{equation*} \label{}
 \left|\oint_{\Gamma_1} Q_n(z)\ P_n(z) \Phi_n (z)\ f(z)\ dz\right|^{1/n } \ \to \
\exp\{ - \min_{\Gamma} (U^{\mu +\sigma} + 2\varphi)  \} =
\exp\{ - \min_{\Gamma} (U^{\mu+ \lambda +  t\nu}\}
\end{equation*}
where $\lambda = \lambda_1 +\lambda_2$ and $\nu = \eta - \lambda_1$.  We multiply this relation by $|c_{n,k}|^{1/n}$ and write it in terms of integrals $I_{n,k}$ defined in \eqref{57} with $G_n = P_n\Phi_n$. Using same arguments as above (following \eqref{57})
to we come to
\begin{equation} \label{59}
 \left|I_{n,1}\right|^{1/n } \ \to \  e^{- m_1(t)} \qquad \text{and} \qquad
\varlimsup_{n\to\infty} \  \left|I_{n,2}\right|^{1/n } \  \leq  \  e^{- m_2(t)}
\end{equation}
where
\begin{equation*} \label{}
m_1(t) =  u_1 + \min_{\Gamma_1 }\  U^{\mu_1+ \lambda +t\nu} \qquad \text{and}
\qquad m_2(t) =  u_2 + \min_{\Gamma_2 }\ U^{\mu_2+ \lambda +t\nu}.
\end{equation*}
According to the definition of $\nu$ and properties of the balayage we
have $U^\nu(z) = c =\const$ on $\Gamma_2$ and $U^\nu(z) = c  -
G^\eta(z)$ on $\Gamma_1$ where $ G^\eta(z)$ is the Green potential of
the measure $\eta$ with respect to the domain $\Omega = \ol {\CC}
\setminus \Gamma_2.$ We have $ G^\eta(z) > 0$ in $\Omega$ and,
therefore, $U^\nu(z) < c$ on $\Gamma_1$. It follows
\begin{equation*} \label{}
m_1(t) =  u_1 + \min_{\Gamma_1 } (U^{\mu_1+ \lambda}  +t U^\nu) <
c +  u_1 + \min_{\Gamma_1 }\  U^{\mu_1+ \lambda}  = c + m_1
\end{equation*}
At the same time we have
\begin{equation*} \label{}
m_2(t) =  u_1 + \min_{\Gamma_1 } (U^{\mu_1+ \lambda}  +t U^\nu) =
c +  u_1 + \min_{\Gamma_1 }\  U^{\mu_1+ \lambda}  = c + m_2
\end{equation*}
Since $m_1 = m_2$ we have  $m_1(t) < m_2(t)$ for $t>0.$ Now relations \eqref{59} contradict orthogonality relations and proof s completed.

\subsubsection{Proof of theorem 1 (general case)}
Generalization of the proof from the case $s=2$ to the general case $s
\geq 2$ is rather straightforward. In short, all the arguments remain
valid with some, mostly obvious, modifications.

First, in case $s > 2$ we have to use more general vector equilibrium
problem which assigns total masses to components of the vector
equilibrium measure according to components of a vector $\vec t = (t_1,
\dots, t_s)$. So, we are going to have several parameters instead of
one. We will consider vectors $\vec t = (t_1, \dots, t_s)$ near the
point $\vec t_0 = (1, \dots, 1)$ with condition $t_1+ \dots + t_s = s.$
Thus, class \eqref{56} of vector measures generalizes to

\begin{equation} \label{60}
{\vec {\mc M}}^{\vec t}={\vec {\mc M}}^{\vec t}(\vec{F}) =
\left\{\vec{\mu}
=(\mu_1,\mu_2, \dots, \mu_s ):\mu_j\in\mc {M}^{t_j}\(F_j\)\right\}, 
\end{equation}
The definitions of $\vec { \Gamma } (\vec t) = \vec\Gamma (\vec t, \vec f),$  ${\vec \lambda}(\vec t)  = (\lambda_1({\vec t}), \dots , \lambda_s({\vec t}))$ (and other parameters) associated the equilibrium problems \eqref{4} - \eqref{8} are modified according to this more general equilibrium problem and become functions of $\vec t$.
In particular, after a sequence $\Lambda \subset \N$ is selected such that \eqref{52} is satisfied, we define
functions
\begin{equation} \label{61}
m_k = m_k(\vec t) =  u_k + \min_{\Gamma_k }\  U^{\mu_k+ \lambda},
\quad \lambda = \lambda_1 + \dots + \lambda_s
 \qquad k = 1, 2, \dots, s.
\end{equation}
where $\Gamma_k$ and $\lambda_k$ are functions of $\vec t$.

Like in case $s=2$ we prove, first. that at $\vec t = \vec t_0$ the
numbers $m_k(\vec t)$ are all equal.  The proof follow same path as in
case $s =2$. We assume the contrary and then, playing with components
of $\vec t$ in a neighborhood of $t_0$ (subject to the condition $t_1+
\dots + t_s = s$) we can find a particular $\vec t$ for which one of
the numbers $m_k(\vec t) $ will be strictly larger that others. Then we
come to a contradiction with orthogonality relations like in case $s =
2$.

Next, under the assumption $ m_1(\vec t_0) = \dots =  m_s(\vec t_0)$ we have to prove that $\vec \mu = \vec \lambda({\vec t_0}).$ Again, we assume the contrary, which means that there is an index, say, $k=1$ such that $\mu_1 \ne \lambda_1.$ From here we need to come to a contradiction.

This part of the proof require just one modification in the choice of measure $\eta$. Instead of the
balayage of $\lambda_1$ onto $\Gamma_2$ in case $s=2$ we define $\eta$ as the
balayage of $\lambda_1$ on the union $\cup_{k=2}^s \Gamma_k$. By this the proof of theorem 1 is completed.


\section{Extremal compact set $\vec\Gamma(\vec f)$ \\
and associated Riemann surface $ \mc R(\vec f)$. }

First, we will prove lemma \ref{lem2} on existence of the extremal compact set $\vec\Gamma(\vec f)$  maximizing the functional of equilibrium energy $\mc{E}[\vec{F}]$  (see \eqref{4} - \eqref{8} in the class $\mc F (\vec f)$ associated with $\vec f.$ Then, we prove that $\vec\Gamma(\vec f)$ has the $S$-property \eqref{14}.


As we noted in the introduction we need rather the $S$-property of $\vec\Gamma(\vec f)$
and lemma \ref{2} is just a convenient way to define $\vec\Gamma(\vec f)$.
It was also noted that in general settings, this extremal compact set in max-min energy problem will not have $S$ property induced by matrix $A$ in \eqref{4}. However,  under the Angelesco condition it is true and it is a direct corollary of known results. We will refer to the paper \cite{Rak12} where the other references may be found.

An assertion of the lemma \ref{lem2} would require some comments. Actually, the lemma is also  a corollary of known facts and techniques but there is no single reference which may be applied. We make a few remarks explaining the reduction in the next section.

\subsection{Proof of lemma \ref{lem2} and $S$-property of $\Gamma(\vec f)$}
The key point in the proof of lemma \ref{2} is continuity of the energy
functional $\mc{E}[F]$ in the Hausdorff metric.

 \subsubsection{Hausdorff metric in a space of vector compact sets.}
To introduce a version of the vector Hausdorff metric on the set $\vec
{\mc F} = \vec {\mc F}(\vec f) $ of vector compact sets $\vec F = (f_1,
\dots, F_s)$ we use the usual scalar Hausdorff metric. For two compact sets
$F_1,F_2\subset \CC$ their Hausdorff distance $\delta_H$ is defined as
\begin{equation}
\label{62}
{\delta}_H\(F_1, F_2\)=\inf\left\{\delta>0: F_1\subset
\(F_2\)_{\delta}, F_2\subset\(F_1\)_{\delta}\right\}
\end{equation}
where $(F)_{\delta}  = \{z\in \CC:\min_{\zeta\in F} |z -\zeta|<\delta\}$ is $\delta$-neighborhood of $F.$

An associated distance $d_H$ between two vector compact sets $\vec F^1$ and $\vec F^2$ with $s$ components is defined as follows
\begin{equation}\label{63}
d_H(\vec F^1, \vec F^2) = \sum_{k =1}^s \delta_H( F_k^1, \vec F_k^2)
\end{equation}
The properties of the metric space of vector compact sets are essentially
same as for the scalar case. In particular, set of all vector compact sets
in a closed disc $\ol D_R = \{z: \ |z| \leq R\}$ of radius $R>0$ is a
compact. The same is true for the class $\vec F \in {\vec {\mc F}}_R$
which consists of vector compact sets $\vec F \in \vec {\mc F}$ such their
components are in $ \ol D_R.$

 \subsubsection{Continuity of vector equilibrium energy in Hausdorff metric.}

We need only a ``point wise'' continuity stated in the following lemma.

\begin{lemma} \label{lem8}
For any $\Gamma \in \vec {\mc F}$ and any $\varepsilon > 0$ there is $\delta > 0$ such that for any $F \in \vec { \mc F}$ with ${\delta}_H\(F_1, F_2\) < \delta$ we have
$|\mc{E}[\vec\Gamma] - \mc{E}[\vec F]| < \varepsilon. $
\end{lemma}
\begin{proof}
The assertion of the lemma is similar to theorem 9.8 in \cite{Rak12}. More exactly theorem 9.8 is the scalar weighted version of the lemma \ref{lem8}. We do not have external field our case. Generalization of the theorem 9.8 to the vector situation is rather straightforward. We need to prove vector versions of lemmas 9.4 - 9.6 in section 9.4 in \cite{Rak12}. To do that we can, for instance, use component wise balayage
for vector measures as follows.

Let $\vec\lambda$ and $\vec\mu$ be equilibrium measures of $\vec
\Gamma$ and $\vec F$ respectively. Let $ \mu'_k$ be the balayage of
$\lambda_k$ onto $F_k$ and let $\lambda'_k$  be the balayage of $
\mu_k$ onto $\Gamma_k$. Then we have estimates
\begin{equation*}
\mc{E}[\vec\Gamma] \leq \mc{E}(\vec {\lambda'})  \leq \mc{E}[\vec F] + \varepsilon/2,
\qquad
\mc{E}[\vec F] \leq \mc{E}(\vec {\mu'})  \leq \mc{E}[\vec\Gamma] + \varepsilon/2.
\end{equation*}\label{}
First inequality in each pair is the extremal property of equilibrium measure. Second inequality in each pair is a vector version of the lemmas 9.4 - 9.6 in section 9.4 in \cite{Rak12} which is obtained by taking sum over components. 
\end{proof}




 \subsubsection{Proof of lemma \ref{2}.}
To complete the proof of lemma \ref{2} we consider a maximizing sequence $\vec\Gamma_n$  in the extremal problem \eqref{6}, that is,
\begin{equation} \label{64}
\mc{E}[\vec\Gamma_n] \to \mc E = \sup_{\vec{F}\in\mc{F}}\mc{E}[\vec{F}].
\end{equation}
First we prove that for any such sequence is bounded. It is enough to prove that the sequence of corresponding supports $\vec\Gamma^1_n$ remains bounded.
We note that actually maximizing sequence may be selected in the convex hull of $e = \cup e(f_k)$.  We need only to show that there is a finite positive $R$ such that $\vec\Gamma^1_n \subset D_R = \{z: |z| \leq R\}$ for large enough $n$ assuming that $e = \bigcup e_k \subset D_{1}$

For an arbitrary $\vec F \in \vec{\mc F}$ we denote as usual by $\vec \lambda = \vec \lambda_{\vec F}$ the extremal (equilibrium) measure in \eqref{6} and further,  $W_k(z) = U^{\lambda_k + \lambda}(z)$ where $\lambda = \lambda _1 +\dots+ \lambda_s$,
$w_k = \min_{z \in F_k} W_k(z)$
(see \eqref{12} and \eqref{13}).

Measure $\mu = \lambda_k + \lambda$ is a positive Borel measures in
plane with total mass $s+1$. If $\mc M$ is set of all such measures
$\mu$ then maximal value for $\min_{z \in F_k} U^{\mu}(z)$ over the
space $\mc M$ is equal to $1/\cop(F)$ (and it is assumed for $\mu =
(s+1)\omega$ where $\omega$ is the Robin measure of $F$). It follows
$w_k \leq \log (1/\cop(F)) \leq C\ $ where $C$ is a constant depending
on $\vec e$ and not depending on $\vec F \in \vec {\mc F}.$

Let $\vec F = \vec \Gamma_n \in \vec {\mc F}$ be a member of maximizing
sequence. By definition of  $\vec {\mc F}$ each component of
$\vec\Gamma^1_n$ consists of finite number of continuums, each one
containing at least two points of corresponding $e_k.$ Now, if (for
some $k$) $\Gamma \subset (\vec\Gamma^1_n)_k$ is such a continuum and there
is a point $z \in \Gamma$ with $|z| \geq R$ then $\cop \Gamma \geq
(R-1)/4 $ and $w_k \leq \log(4/(R-1)).$ On the other hand, it follows
from definitions that for any $\vec F  \in \vec {\mc F}$ we have  $\mc
E[\vec F] = w_1 +\dots + w_s$. Thus, if sequence $\Gamma_n$ is not
bounded then we have $\mc{E}[\vec\Gamma_n] \to -\infty$ (along some
subsequence) and, therefore, $\mc E =
\sup_{\vec{F}\in\mc{F}}\mc{E}[\vec{F}] = - \infty$ which is a
contradiction since $\mc E $ is evidently finite.

Finally, combining the assertions made above we conclude that some
subsequence of minimizing sequence $\vec\Gamma_n$ converges in
Hausdorff metric to a vector compact set $\Gamma \in \vec{\mc F}$ and by
continuity of $\mc E$ we have $\mc E(\vec \Gamma) = \mc E $ (see
\eqref{64}). This complete the proof of lemma \ref{2}.

\subsubsection{$S$-property of $\vec \Gamma$.}
Each component $\Gamma_k$ of the extremal compact set $\vec \Gamma$ defined by the extremal problem \eqref{8} is a solution of a scalar weighted problem if we assume that the other components are fixed. For instance, with respect to variations of the first component ($k=1$) we have the following extremal property of $\Gamma_1$
\begin{equation*} \label{}
\mc{E}[\vec{\Gamma}]= \mc{E}[(\Gamma_1, \Gamma_2, \dots, , \Gamma_s)]=
\max_{F_1\in \mc F(f_1)}\mc{E}[(F_1, \Gamma_2, \dots, , \Gamma_s)],
\end{equation*}
and the same is true for $k =2, \dots, s$ (follows directly from definitions). Also, the $k$-th components $\lambda_k$ of the vector equilibrium measure $\vec \lambda$ provides the minimum for the total energy
in class $\mc M(\Gamma_k)$ for the fixed other components. The equilibrium energy
$\mc E(\vec\lambda)$ as function of the $k$-th component $\lambda_k$ is represented as
\begin{equation*} \label{}
\mc E(\vec \lambda) = \mc E_{\varphi_k}(\lambda_k) + C_k
\qquad \text{where} \quad  \varphi_k(z) =  \frac 12 \sum^s_{i\ne k } U^{\lambda_i}
\end{equation*}
and $C_k$ does not depend on $\lambda_k.$

Under the Angelesco condition external field $\varphi_k$ is harmonic in a neighborhood of $\Gamma_k$
and the $S$ property of $\vec \Gamma$ with respect to $k$-th coordinate follows by the theorem 3.4 in \cite{Rak12}.\\


\subsection{Extremal compact set in scalar case $s=1$.}  
Representation of $\vec\Gamma(\vec f)$ as a whole or its components is
an interesting and complicated problem. It is connected to many other
problems and altogether they constituting a field in classical complex
analysis. We will review a few particular results in the field related
to study of $\vec\Gamma(\vec f)$ the Angelesco situation. The case is
much simpler then the general one and may be essentially reduced to
weighted scalar case $s =1$. Indeed, each components $\Gamma_k$ is the
scalar extremal compact set in the harmonic external field $\varphi_k$
generated by potentials of equilibrium measures of other components.

We will go into some details related to the case  and we will begin
with nonweighted situation when Hermite--Pad\'e polynomials become
Pad\'e polynomials, that is numerator and denominator of diagonal
Pad\'e approximants at infinity for a single function element $f = f_1
\in \mc A$ not depending on $n$.


\subsubsection{Stahl's theorem. Extremal compact set $\Gamma(f)$. }
In case $s=1$ definition \eqref{2} with $p_n = - q_{n,0}$ and $q_n = q_{n,1}$  become
\begin{equation*}
R_n(z):= (q_n f - p_n)(z) = O\(\frac1{z^{n + 1}}\),
\label{}
\end{equation*}
 The rational function $\pi_n = p_n/q_n$ is the (diagonal) Pad\'e approximant to $f$ at infinity of order $n$

One of the main problems in the theory of Pad\'e approximants in 1960 -
1970-th was the convergence  problem for function $f\in  \mc A.$ The
problem is essentially equivalent to the problem of zero distribution
of denominators $q_n$ and it may be viewed as a particular case the
zero distribution problem for Hermite--Pad\'e polynomials. First
results in this direction were obtained by J. Nuttall (for functions
with quadratic branch points) who also made a general conjecture (see
\cite{Nu77} and \cite{Nu84}) for any  $f\in  \mc A$  we have
\begin{equation*}
\label{}
\pi_n\overset{\cop}{\to}f,\quad z\in\mbb{C}\setminus \Gamma(f),
\quad \text{where}\quad
\cop(\Gamma(f))=\min_{F \in \mathcal F }\cop(F)
\end{equation*}
where class $\mc F(f)$ of admissible cuts is same as above (see introduction) and $\overset{\cop}{\to}$ stands for convergence in capacity (on compact sets in the indicated domain).

General theorem on convergence of Pad\'e approximants, including the Nuttall's conjecture, has been proved by
H.~Stahl \cite{Sta85a}--\cite{Sta86b} (methods used in this paper are in part originated there). In particular, he proved the zero distribution formula
$$\frac 1n \mc X(q_n) \to \lambda \qquad \text{ where} \quad \lambda=\lambda_\Gamma$$
is the Robin (equilibrium) measure of the extremal compact set $ \Gamma(f)$. The (negative) equilibrium potential is (up to a constant) the Green function $g(z)$ for the complement to the extremal compact.
$$ w - U^\lambda(z) = g(z), \quad z \in \Omega = \overline\CC\setminus \Gamma(f)$$
Stahl obtained also a rate of convergence $\pi_n\overset{\cop}{\to}f.$ 

Now, we will consider in some details the geometric component of the
theorem. In other words we are interested in the geometric structure of
the minimal capacity compact set $\Gamma$. We restrict ourself to the case
of finite sets $e$ as we generally do in this paper. Assumption of the
original theorem Stahl's theorem in \cite{Sta86a}--\cite{Sta86b} was  $\cop e =
0$ what
is essentially more general and accordingly less constructive.
Characterization in terms of quadratic differential (see below) is
still valid, but becomes more complicated (differential is not
rational). The associated Riemann surface may also be introduced but
will not be closed. Here we do not discuss general case.

Minimal capacity property of $\Gamma(f)$ is equivalent to maximal
equilibrium energy, so that $\Gamma$ is exact scalar analogue of our vector
compact set $\vec \Gamma.$ Hence, the study of $\Gamma(f)$ would be the
first step in the study of the geometry of the vector compact set
$\vec\Gamma (\vec f).$ In scalar case $s=1$ without external field there is
a well developed theory.

Many part of this theory may be generalized, one way or another, to the
weighted case and, then, to the vector case $s > 1$ (for the weighted
case~\cite{GoRa87} and also~\cite{BaYa09},~\cite{BaStYa12},~\cite{Bus13},~\cite{Bus15a},~\cite{Bus15b}.)
However,
generalizations are not always obvious and very often exist only as
conjectures; see~\cite{Sta88},~\cite{Sta12}. Theory for the vector case does not exist yet; there are
several separated fragments. Case $s=1$ may serve as a good introduction to
the matter.

\subsubsection{ Extremal compact set $\Gamma(f)$ in terms of quadratic differential. }
The minimal capacity problem in class $\mc F(f)$ for finite $e(f)$ is a
direct generalization of a classical Chebotarev problem on continuum of
minimal capacity containing set $e$. Solution of the problem and its
characterization in terms of critical trajectories of a quadratic
differential was well known since 1930-th; see references and details
in \cite{Str84}. Solution in class $\mc F(f)$ is essentially similar
and is presented in the following lemma.

\begin{lemma} \label{lem9}
Let $f \in \mc A$, $e(f) = \{a_1, a_2, \dots, a_p\}$ and  $A(z) = (z-a_1)(z- a_2) \cdots(z - a_p).$ Then there exists a polynomial
\begin{equation*}
V_f(z) = (z-v_1) (z-v_2) \dots (z-v_{p-2}) \qquad \text {where}\quad v_j = v_j(f).
\label{}
\end{equation*}
of degree $p-2$ depending on $f$ such that the extremal compact set $\Gamma(f)$ is a union of some of critical trajectories of the quadratic differential $ - (V/A)\,(dz)^2$ where $V = V_f$.

Moreover, $ - (V/A)\,(dz)^2$ is the quadratic differential with closed trajectories.
\end{lemma}

For a proof see \cite{Sta85a}; see also \cite{Str84}. An alternative proof based on max-min energy problem was presented in \cite{PeRa94}; see the review \cite{Rak12} for further details.

Connection of extremal compact sets with quadratic differentials is
fundamental. In particular, it allow us to introduce a reach differential
geometric context for the potential theoretic max-min energy problem. From
here we may obtain a number of equivalent reformulation of the problem (for
instance, in terms of moduli of families of curves) and this is a large
source of methods; see~\cite{Gol66}, \cite{Ku80},  \cite{Str84}, \cite{MaRa10},
\cite{Rak12} for a general discussion and for further references.

We will make a few short remarks extending assertions of lemma \ref{lem9} and showing to some extend a larger content.

First, we comment on quadratic differential with closed trajectories. Trajectory is a (maximal) curve $\gamma$ such that
\begin{equation*} \label{}
\frac{ V(z)}{A(z)}\  dz^2 < 0 \quad \text{or} \qquad \Re\int_a^z \, \sqrt{ {V(t)}/{A(t)}}\, dt = \const
\qquad \text {on}\quad  \gamma
\end{equation*}
Trajectories are also some particular geodesics of the metric  $ |V/A)|\,|dz|^2$ in plane. All such curves are analytic. Critical trajectories are analytic arcs connecting two points from the set of zeros of $AV$. For a {\it quadratic differential with closed trajectories} each noncritical trajectory is closed. In this case there is a finite signed (real) measure $\sigma$ such that
\begin{equation*} \label{}
U^\lambda (z) = \Re \int^z \, \sqrt{ {V(t)}/{A(t)}}\, dt  \qquad \text{or} \qquad
C^\lambda (z) : =\frac1{\pi i}\int_\gamma \frac{d\lambda (x)}{z - x}  = \sqrt{ {V(z)}/{A(z)}}
\end{equation*}
Any such measure $\lambda$ is a critical point of the logarithmic energy functional with respect to local variations with fixed set $e$ (signed critical measure). Also, potential of any such measure is a constant on any connected components of the support.

 For a given $A$ there is a large set of polynomials $V$ such that  $ - (V/A)\,(dz)^2$ is a quadratic differential with closed trajectories; such polynomials $V$ are dense in $\P_{p-2}$. Consequently, there is a large set of signed critical measures. More important for our current purposes are positive critical measures.

For a given $A$ there is a $p-2$-parametric family of polynomials $V$
such that  $ - (V/A)\,(dz)^2$ is a quadratic differential with closed
trajectories and associated measure $\lambda$ is positive. Using zeros
$v_i$ as parameters, this family may be represented as a union of
$3^{p-2}$ analytic bordered manifolds (cells). Polynomials $V$
associated with Stahl's compact sets $\Gamma(f)$ are included in this
family (as corner points of cells). Measure $\lambda$ in this case is
the Robin measure of $\Gamma(f)$;  see \cite{MaRa10} for further
details.

Next we consider some details related to the set of $V$-polynomials originated by extremal compact sets $\Gamma(f)$ of all functions $f \in \mc A$ having the same branch set $e$ and, hence, associated with the same polynomial $A$.

\subsubsection{Chebotarev's continuum. Set $\widehat {V}_e $. }
Lemma \ref{lem9} would serve as a constructive characterization of
$\Gamma(f)$ if we may determine corresponding polynomial  $V(f).$  Here
we discuss briefly a combinatorial component of the problem.

Polynomial $V(f)$ depends, first of all, on the branch set $e$ of
function $f\in\mathcal A$. It depends also on branch type of the
function. We say that two functions elements $f, g  \in \mc A(\ol {\CC}
\setminus e)$ at infinity have the same branch type if after analytic
continuation along any loop in $\ol {\CC} \setminus e$ both elements $f$
and $g$ remains unchanged or both change. If $f$ and $g$ have the same
branch type  then $\Gamma(f) = \Gamma(g)$. The inverse is not true,
functions with different branch type may have the same extremal compact
since not all the loops turns to be important.

Anyway, dependence of $\Gamma(f)$ on branch type of $f$ is reduced to a finite number of options. To make it formal, for a fixed $e$ consider the set of all functions $f \in \mc A(\ol {\CC} \setminus e)$. Then set of all associated extremal compact sets $\Gamma(f)$ is finite and the set
\begin{equation*} \label{}
\widehat {V}_e = \left\{ V (f) : \,\, f \in \mathcal A\,(\ol {\CC} \setminus e)  \right\}
\end{equation*}
of corresponding polynomials $V(f)$ is also finite (it is not entirely on the surface, but is still a simple corollary of known results). As a remark, the number of elements in this sets depend on location of points $a \in e$. It is easy to calculate the maximal (for given $p$) number $m_p$ for small number $p = \#(e)$; we have $m_2 = m_3 =1$, $m_4 = 2$ and $m_5 = 3$ (in situation of common position $m_p$ is equal to the actual number elements in $\widehat {V}_e$). Starting with $p =6$ this counting become more complicated. A general approach to this combinatorics may based of the analysis of the Chebotarev's continuum associated with the set $e$ which is.

Chebotarev's continuum  $\Gamma_e$ for a (finite) set $e$ is the continuum of minimal capacity in class of continua containing the set $e.$ The existence and characterization problem for such a continuum is known in the geometric function theory as Chebotarev's problem.  It was solved independently by Grotsch and Lavrentiev in 1930-th; see \cite{Str84} for details.

In particular, lemma \ref{lem9} were long known for the solution of Chebotarev's problem. Polynomial $V = V_e$ corresponding to $\Gamma_e$ is a particularly important element of the set $ \widehat {V} (e)$ since it may be used for construction of the other Stahl's compact sets associated with the same set $e$.

The structure of the Chebotarev's continuum $\Gamma_e$ depends on the
configuration of points in $e$ and in general also contains a
nontrivial combinatorial component. It is not our purpose here to
present a complete analysis of the situation; assume for simplicity
that the polynomial $V_e$ has simple zeros. This constitute a situation
of ``a common position'' for points $a_j \in e$ (configurations
$\{a_k\}$ not satisfying this condition have positive co-dimension).
Then the continuum $\Gamma_e$ is a union of $2p-3$ analytic arcs
$\gamma_k$
\begin{equation*} \label{}
\Gamma_e\  = \ \gamma_1\cup  \gamma_2 \cup \dots \cup \gamma_{2p-3}
\end{equation*}
Interiors parts of arcs $\gamma_k$ are disjoint and their end points
belong to the set of roots of the polynomial $A(z)V_e(z)$. Each $a \in
e$ is the end point of a single arc, each root $v$ of the polynomial
$V_e$ is common endpoint of three arcs.  We will say that $\gamma_i$ is
$a$-$v$ arc if one of its endpoints belong to $e$ and another one is a
zero of $V$. If both endpoints of an arc are zeros of $V$ we call it a
$v$-$v$ arc (the are no $a - a$ arcs in Chebotarev's continuum).
Totally, $\Gamma_e$ consists of $p$ arcs of type $a$-$v$ and $p-3$ arcs
of $v$-$v$ type. The graph (tree) generated by this collection of arcs
may serve as a definition of the combinatorial structure of $\Gamma_e$.
If this structure is known then points $v_i$ are uniquely defined by
the system of equations
\begin{equation} \label{Equ}
 \Re \ \int_{\gamma_k}\sqrt{\frac {V(t)}{A(t)}}\, dt\  =\  0, \qquad k = 1, 2, \dots, 2p - 4
\end{equation}
(with a the proper choice of orientation, sum of integrals over all $2p-3$ arc is equal to $\pi i$ , so that equality for $k = 2p -3$ is the corollary of the others).

All the other elements of the set $\widehat {V}_e$ different from
$\Gamma_e$ may be obtained using the following procedure of ``fusion'' of
connected zeros of $V_e$. We select any  $v$-$v$ arc in the Chebotarev's
continuum and eliminate this arc from $\Gamma_e.$ This will divide
$\Gamma_e$ into two disjoint connected components and, as a corollary, we
obtain a partition of set $e$ into two subsets $e = e_1 \cup e_2$. Using
this partition we introduce a modified minimal capacity problem in class of
compacta $\mc F = F_1 \cup F_2$ where $F_i$ is a continuum containing $e_i$
where $ i =1, 2$. The solution $\widetilde\Gamma$ of this problem will have
the form $\widetilde\Gamma = \Gamma_1 \cup \Gamma_2$ and it will be one of
Stahl's compacta associated set $e$. Corresponding polynomial $\widetilde V
\in \widehat {V}_e$ will have (at least) one double zero replacing the
original $v$-$v$ pair. The other zeros of $V$ may be put in correspondence
with the remaining zeros of $V_e$ (any polynomial $V \in \widehat {V}_e$
different from  $V_e$ have multiple zeros). Then this operation may be
repeated until no $v$-$v$ arcs are left.


It is also possible to describe modification of the system \eqref{Equ} needed to pass to corresponding system of equations defining the roots of $\widetilde V \in \widehat {V}_e;$ we will mention only that number of $v$ parameters is now $p-3$ (one of them is marked as a double root); consequently we will have less by one number of equations.

The fusion of a pair $v$-$v$ into a double zero may be carried out
continuously, using an intermediate critical measures. More exactly,
the finite set of Robin measures  associated with Stahl's compact sets into
a finite dimensional variety of  $e$-critical measures  (method was
presented in \cite{MaRa11}). Zeros of the extended family of
polynomials $V$ play role of coordinates in the family of critical
measures and the Chebotarev continuum $\Gamma_e$ may be viewed as the
origin in this coordinate system.



The system of equations \eqref{Equ} show that locally real numbers $\Re v_i, \  \Im v_i $ are real analytic functions of $\Re a_k, \  \Im a_k $. It is known that dependence $v_i(a_k)$ is globally continuous.

\subsubsection{Green function for the domain $\  \Omega = \overline\CC\setminus \Gamma(f)$  }
Next we review basic facts related to an alternative characterization
of compact sets $\Gamma(f)$ for $f \in \mc A$ in terms of $g$ functions of
certain family of hyperelliptic Riemann surfaces. A $g$ function may be
introduced as a harmonic continuation of the Green function of the
complement to $\Gamma(f)$ (with pole at infinity). In turn, the Green
function $g(z)$ is reduced to the equilibrium potential as follows
\begin{equation} \label{Gre}
g(z) = w - U^\lambda(z), \quad  z \in \Omega = \overline\CC\setminus \Gamma(f)
\end{equation}
where $w$ is the equilibrium constant so that $g = 0$ on $\Gamma$. Function $g$ is harmonic in finite part of $\Omega$ and $g(z) = \log |z| + O(1)$ at infinity. These are characteristic property of Green function of an arbitrary regular domain containing infinity.

For the Green function $g(z)$  of the complement to an extremal compact set $\Gamma(f)$ and associated  complex Green function  $G(z) = g(z) + i \tilde g(z)$ an explicit formulas follow from lemma \ref{lem9}. If $V$ is associated polynomial then we have   $G'(z) = \sqrt{ {V(z)}/{A(z)}}. $ 
Hence,  
\begin{equation} \label{Gre_n}
g(z) = \Re G(z),\qquad  G(z) = \int_a^z \, \sqrt{ {V(t)}/{A(t)}}\, dt \quad (a\in e),
\end{equation}
(branch of the root is such that $g(z) = \log|z| + O(1)$ at infinity).
In particular, it follows from here that the extremal compact set $\Gamma$
has the $S$-property.  In terms of the Green function $g$ it is stated
as follows
\begin{equation} \label{S}
\frac{\partial g}{\partial n_1}(\zeta)
=\frac{\partial g}{\partial n_2}(\zeta)
,\quad \zeta\in\ \Gamma
\end{equation}
Since $g(z) =  w - U^\lambda(z) $ the same may  be written for the equilibrium potential $U^\lambda$ (compare to \eqref{14}).

The $S$-property \eqref{S} plays an important role in Stahl's approach to asymptotics of
complex orthogonal polynomials (Pad\'e denominators).  In terms this property Stahl defined the extremal compact set $\Gamma(f)$. The compact set $\Gamma(f) \in \mc F(f)$ is the unique compact set with the $S$-property such that jump $(f^+ - f^-)(\zeta) \not\equiv 0$ on any analytic arc in $\Gamma(f)$.

Thus, in this case the definition in terms of $S$ property is equivalent to the definition in terms of max-min energy problem. The same is true in many other cases.  A general approach to the existence problem for $S$ curves in given class may based on this equivalency is presented in \cite{Rak12}. We note, however, that $S$-property is more general; it characterize rather arbitrary critical point of energy then (local) maxima of equilibrium energy.

\subsubsection{Riemann surface $\mc R(f)$. }
Now, from the domains $\Omega$ and their Green functions we go to a
Riemann surfaces and their $g$-functions. Characterization of $S$
compact sets in terms of $g$-functions was used by J.~Nuttall \cite{NuSi77}
for a particular case of functions $f$ with quadratic branch points.
The case when branch points are real goes back to N.~I.~Akhiezer
\cite{Akh62}. More general approach was outlined in \cite{Rak16}. Here
we make a few remarks following presentation  in \cite{Rak16}.

The Green function $g$ of the domain $\Omega$ has a harmonic
continuation to the hyperelliptic  Riemann surface ${\mathcal R} =
\mathcal R(f)$  of the function $\sqrt{AV}$. We interpret $\mathcal R$
in a standard way as a (two sheeted) branched covering over $\overline
\CC$ with canonic projection $\pi: \mc R \to \ol{\CC}$.This fact follows
directly from \eqref{Gre}.

At the same time the assertion on extension of $g$ from $\Omega$ to
$\mc R$ may be derived from the $S$-property \eqref{S} combined with
boundary condition $g(z) = 0$ for $z \in \Gamma$. The first step of the
proof is the construction of $\mc R$ based on a standard procedure of
gluing of $\mc R$ from two copies $\Omega_1$ and $\Omega_2$ of $\Omega$
closed by adding the set of accessible boundary points. Formally such a
closure may be defined by introducing the inner metric (the same in
each copy)
\begin{equation}\label{Dist}
dist(z, \ \zeta) = \inf\ \{ \ell (\gamma): \gamma \subset \Omega_k,  \quad z, \ \zeta \in \gamma \}
 \end{equation}
where $k = 1,2$ and $\ell(\gamma)$ is the length of a curve $\gamma$. The closure of $\Omega_k$ in this metric
\begin{equation} \label{Clo}
\overline \Omega_k = \Omega_k \cup \Gamma_k^+ \cup \Gamma_k^-
\end{equation}
contains two identical copies of the topological boundary $\partial \Omega_k = \Gamma_k.$

Next, we introduce the equivalency relation $\sim$ in the topological sum $\Sigma = \coprod_{k=0}^s \overline \Omega_k$ identifying $\Gamma_1^+$ with $\Gamma_2^-$ and $\Gamma_2^+$ with $\Gamma_1^-$. (each interior point is equivalent only to itself). Then we define $\mathcal R = \Sigma / \sim$  as a quotient space with respect to this equivalency. Local coordinates and canonic projection $\pi(\zeta)$ are defined in a standard way. Thus, construction of $\mc R$ is completed.

Next, starting with the Green function $g_\Omega(z)$ of $\Omega$ we define the $g$-function on $\mc R$
\begin{equation} \label{g}
g(\zeta) = g_\Omega(\pi(\zeta)), \quad z \in \Omega_1 , \qquad
g(\zeta) = - g_\Omega(\pi(\zeta)), \quad z \in \Omega_2
\end{equation}
and $g(\zeta) = 0$ for any $\zeta \in \mc R$ with $\pi(\zeta) \in \Gamma$.

Continuity of $g$ on the finite part of $\mc R$ follows by continuity of $g_\Omega(z)$ in $\CC$. In addition to that the $S$ property \eqref{S} combined with definition of $g$ in\eqref{g} show that the gradient of $g$ is continuous in finite points of $\mc \R$ whose projections are not in $e$. It follows that $g$ is harmonic in $\mc R \setminus \pi^{-1}(\infty)$.

Independently of the Green function of the domain $\Omega$, the
$g$-function of an arbitrary hyperelliptic Riemann surface $\mc R$ (not
branched at infinity) is defined as a unique real valued harmonic
function $g(\zeta): \mathcal R\setminus \pi^{-1}(\infty) \to \overline
\R$ on the finite part of $\mathcal R$ with the following behavior at
infinity
\begin{equation} \label{G-1}
\begin{aligned}
& g(\zeta) =  \log|z| + O(1), \quad z \to \infty^{(1)}, \quad z = \pi (\zeta), \\
&g(\zeta) = - \log|z| + O(1), \quad z \to \infty^{(2)},
\end{aligned}
\end{equation}
and normalization $g(z^{(1)}) + g(z^{(2)}) \equiv 0$ (note that $g(z^{(1)}) + g(z^{(2)}) \equiv \const$
according to \eqref{G-1}). Corresponding complex $G$ function is the third kind Abelian integral with
two marked points $\zeta_1 = \infty^{(1)}$ and $\zeta_2 = \infty^{(2)}$ and divisor indicated in \eqref{G-1}. The differential $dG(\zeta)$ is associated third kind Abelian differential.

For the Riemann surface associated with an extremal compact set $\Gamma(f)$
the compact set itself is the projection of zero level $\{\zeta: g(\zeta) =
0 \} \subset \mathcal R$ of $g$-function on the plane $\overline \CC$.
Associated complex function $G(z)$ is multivalued analytic function of
$\mc R$ with real part $g(z).$

Formulas \eqref{Gre} remain valid.


Now we observe that the form of the $g$-function for $\mc R(f)$ in \eqref{Gre} is different from the generic form of $g$-function of a hyperelliptic Riemann surface.
 Indeed, consider the $g$-function associated with a generic hyperelliptic Riemann surface $\mc R$.
Let branch points of $\mc R$ be (distinct) zeros of a polynomial $X(z)
= \prod_{i =1}^{2p}(z - x_i)$ ($\deg X$ is even, so that
$\infty^{(1,2)} \in \mc R$ are not branch point). The $g$-function for
$\mc R$ has representation
\begin{equation} \label{Gre-1}
g(z) = \Re G(z),\qquad  G(z) = \int_{x_1}^z \, \frac {Y(t)}{\sqrt{X(t)}}\, dt \quad
\end{equation}
where polynomial $Y(z) = \prod_{i =1}^{p-1}(z - y_i)$ is uniquely determined by $X(z)$ (at this point we assume that zeros of $X$ are simple).

The  formula \eqref{Gre} defines a mapping $X \to Y$ from $\P_{2p}$ to
$\P_{p-1}$. More exactly this formula defines the mapping $X \to Y$ for
a polynomial $X$ with simple zeros. Then, the mapping may be extended
by continuity to polynomials with multiple zeros. It is easy to verify
that the extension has the following property. If $x_0$ is one of zeros
of $X$ of multiplicity $2m$ or $2m +1$ where $m \in \N$ then $Y(z)$
will have zero of multiplicity $m$ at $x_0.$ Continuity of the mapping
$X \to Y$ in a neighborhood of a polynomial $X \in \P_{2p}$ with
multiple zeros is preserved.

Next, we apply the mapping defined above for $X = AV$ where $A = z^p + \dots$  is a fixed polynomial with simple zeros (which come from $e(f)$) and a variable $V(z) = \prod_{i =1}^{p-2}(z - v_i)$. This defines another mapping
\begin{equation} \label{Phi}
T: V \in\PP_{p-2} \ \to \ Y \in\PP_{p-2}
\end{equation}
Now we characterize $\ \widehat {V} (e)\ $ in terms the mapping $T.$
\begin{thm} \label{FP}
The set $\ \widehat {V} (e)\ $  coincide with the set of fixed points
of the mapping $T$
\end{thm}
\begin{proof} The fact that for any $V \in  \widehat {V} (e)$ we have $T(V) = V$ follows directly from
comparing \eqref{Gre-1} and \eqref{Equ}.

Backward, for any $V(z) = z^{p-2} + \dots$ with $T(V) = V$ there exist  $ f \in \mc A(\ol {\CC} \setminus e)$ with $V = V(f)$. We may, for instance, construct such $f$ as follows.

We know that $g$ function for the Riemann surface of the function
$\sqrt{AV}$ has the form \eqref{Gre}. Let $\Gamma$ be the projection of
the zero level of $g$ onto the plane. The set $\Gamma$ is a compact
with $S$ property containing $e' \subset e$ (we can have cancelations,
but $e'$ contains at leat two points). $\Gamma$ has finite number of
connected components which defines partition of $e'$ into groups which
belong to the same connected component of $\Gamma$. Partition of $e'$
defines factoring of correspond polynomial  $A$  into a product of
polynomials $A_i(z)$ whose zeros belong to the same connected component
of $\Gamma$.

In turn, using this factoring of $A$ we define a function $ f \in \mc A(\ol {\CC} \setminus e)$ by $f(z) = z^{m} \prod_{i=1}^m A_i^{-1/d_i}(z)$ where $d_i =\deg A_i.$ Stahl's compact set $\Gamma(f)$ of the function $f$ is $\Gamma$.
\end{proof}

In an equivalent form the assertion of the theorem \ref{FP} may be stated as follows.  For a given $A$ and any $V \in  \widehat {V} (e)$ corresponding Riemann surface $\mathcal R$ satisfies the following property. The derivative $G'(z)$  of the (complex)  $G$-function for the Riemann surface $\mathcal R$ of the function $\sqrt{V/A}$ may have poles only in set $A$.
Backward, any such polynomials $V$ belongs to $ \widehat {V} (e)$.


We can make a summary of the above considerations related to the case
$s = 1$ as follows. For a given function $f \in \mc A\ (\ol {\CC}
\setminus e)$ there exist a unique extremal compact set $\Gamma(f)$ whose
Robin measure represent limit zero distribution of Pad\'e polynomials
for $f$.

By lemma \ref{lem9} this compact set is determined by a unique pair of
polynomials $A(z)$ (representing $e(f)$) and $V = V(f) \in  \widehat
{V}(e)$; different $V(f)$ represent different branch types of functions
$f$ with the same branch sets. Finally, Riemann surface $\mathcal R =
\mathcal R(f)$ of the function $\sqrt{V/A}$ and corresponding
$G$-function are also uniquely defined by $f$ (pair $A,V$).

This establishes, one-to-one correspondence between each two of the following three sets associated with a given set $e = \{a_1, \dots, a_p\}$ of distinct points in plane

(1) set of $S$-compact sets $\Gamma(f)$ for all functions $f \in \mc A\ (\ol {\CC} \setminus e)$,

(2) set of polynomials $\widehat {V}(e)$ and

(3) the set of Riemann surfaces
$
\widehat{\mathcal R}(e) = \left\{\mathcal R(f) : \,\, f \in \mc A(\ol {\CC} \setminus e)\right\}.
$

In the last section of the paper we present a conjecture which is
generalizing to some extent these fact to   the vector case $s >1$. It
turns out that in this case it is more convenient to analyze situation
in terms of set of Riemann surfaces generalizing  set
$\widehat{\mathcal R}(e).$ For the Angelesco case the conjecture is
supported by the main result of this paper.

\subsubsection{Weighted case $s =1$. Extremal compact set in an external field.}
As we have mentioned above, each component of the extremal vector
compact set $\vec \Gamma(\vec f)$ is itself a scalar extremal compact set in
the external field $\varphi$ induced by the equilibrium measures of all
other components. Thus, we can pass to vector case using a weighted
version of the scalar problem.

Generally speaking, presence of the external field may create a new and
significantly more complicated situation. However, in the study of
vector Angelesco case we meet only such weighted situation when the
effect of the presence of the external field is rather mild. In
essence, all what was said above with regard to nonweighted case
remains valid, may be in a somewhat modified form. We discuss briefly
what needs to be changed.

Settings of the weighted problem are the following. Together with
functions  $f \in \mc A (\ol \CC \setminus e)$ where $e = \{a_1, \dots,
a_p\}$ we consider a simpliconnected domain $\mc D$ containing $e$ and
a harmonic function $\varphi(z)$ in $\mc D$.

For any (unit) measure in $\mc D$ we define weighted energy $\mc E_{\varphi}(\mu)$ according to  \eqref{M15}. For any compact set $F \in \mc F(f)$ which is contained in $\mc D$ we define the (unit, weighted) equilibrium measure $\lambda_F$ by \eqref{M16}.

Now, our primary assumption is that for any  $f \in \mc A (\ol \CC \setminus e)$ there exist a compact set $\Gamma = \Gamma(f, \varphi) \in  \mc F(f)$ with
\begin{equation} \label{E}
\mc{E_\varphi}[{\Gamma}]=\max_{{F(f)}\in {\mc{F}}}\mc{E_\varphi}[{F}],
\qquad \text{where} \quad \mc{E_\varphi}[{F}] = \mc{E_\varphi}(\lambda_F)
\end{equation}
We will be eventually interested in case when $\Gamma$ is one of components of Angelesco vector compact set $\vec \Gamma(\vec f)$ and the field $\varphi(z)$ is the potential of other components of the vector equilibrium measure. In this case our assumption on existence of extremal compact set in \eqref{E} will become a corollary of Angelesco condition. In general weighted settings such compact set may not exist and may not be unique, if exists; see \cite{Rak12} for details.
\begin{lemma} \label{lem10}
Under the assumptions above on the extremal compact set $\ \Gamma = \Gamma(f, \varphi) \in  \mc F(f)$ the equilibrium measure $\lambda = \lambda_\Gamma$ satisfies the condition
\begin{equation}
B(z) : = \ A(z) \left( \int \frac{d\lambda(x)}{x -z} + \Phi'(z) \right)^2 \ \in H(\mc D)
\label{B}
\end{equation}
where $\Phi(z) = \varphi(z) + i\tilde \varphi (z) \in H(\mc D) $ is the analytic function in $\mc D$ with real part $\varphi(z)$ 
Moreover,  the extremal compact set $\Gamma(f)$ is a union of some of critical trajectories of the quadratic differential $ - (B/A)\,(dz)^2.$
\end{lemma}
The a proof see \cite{MaRa10} where representation of the lemma is
obtained for critical (so called $A, \varphi $-critical) measures. The
fact that the equilibrium measure of extremal compact set  $\Gamma(f,
\varphi)$ is critical is explained in \cite{Rak12}.




The formal investigation of the combinatorial structure of $\Gamma(f,
\varphi)$ under general assumptions on $\varphi$ is not the purpose of
the paper. We are interested in the case related to Angelesco vector
equilibrium problem. Informally speaking, the Angelesco condition
implies that the external field is harmonic in a neighborhood of
$\Gamma$ and that it is small enough.

Under these conditions the combinatorial structure of $\Gamma(f,
\varphi)$ remains similar to structure $\Gamma(f) = \Gamma(f, 0)$ and,
in particular, there is exactly $p-2$  zeros of $B$ involved in this
structure. To make it more certain let us assume that $f$ has a common
branch type, that is, $\Gamma(f, \varphi)$ is a continuum. Then there
is such a neighborhood of this compact set where $B$ has exactly $p-2$
zeros and each of them is an endpoint of one of the analytic arcs
constituting $\Gamma(f, \varphi).$  Number of arcs is at most $2p-3$
and is equal to $2p-3$ if all those zeros of $B$ are simple.

For a formal proof of the last proposition it is more convenient to use
Riemann surfaces and their $G$ functions (in particular, quadratic
differential $(dG)^2$ is the proper analogue of $(V/A)dz^2$). Here we
mention briefly a homotopic argument which may also be used for formal
proofs. Assume, again, that $\Gamma(f, \varphi)$ is a continuum. We
introduce a parameter $t \geq 0$ and consider the family of vector
fields $t\varphi(z)$ and corresponding compact sets $\Gamma_t = \Gamma(f,
t\varphi).$ Let $B_t(z)$ be corresponding $B$-function in \eqref{B}.
For $t=0$ we are back in nonweighted case and we know that $B_0(z) =
V(z) \in \P_{p-2}$. Thus, $B_0(z)$ has exactly $p-2$ zeros in a
simpliconnected neighborhood $\mc D$ of $\Gamma$. If for any $\tau \in
[0, t]$ the compact set $\Gamma_\tau$ belongs to the domain of harmonicity
of the external field, then zeros of $B$ are still defined by system of
equations \eqref{Equ} with $V$ replaced by $B.$ It follows from these
equations that short trajectories of the quadratic differential  $ -
(B/A)\,(dz)^2$ constituting $\Gamma_\tau$ are preserved and change
continuously as functions of $\tau \in [0, t].$

Thus, $\Gamma_\tau$ and, in particular, zeros of  $B_\tau(z)$ continuously depend on $\tau$ (actually this dependence is real analytic). Note that the combinatorial structure of $\Gamma_\tau$ may change (bifurcation points are values of $\tau$ where $B$ has multiple zeros) but number of parameters involved in the structure remains constant.

Using the arguments above we can calculate number of parameters
involved in the structure of the vector compact set $\vec \Gamma (\vec f)$.
For instance if $s = 2$

\subsection{Riemann surface $ \mc R(\vec f)$.}
The asymptotics of Angelesco Hermite Pad\'e polynomials associated with
the vector function $\vec f \in \mc A$ in \eqref{1} may be described in
terms of $g$-function of an algebraic Riemann surface $\mathcal R = \mc
R(\vec f)$. First, we introduce $g$-function associated with a generic
Riemann surface defined as a branched covering over the sphere.

\subsubsection{$G$-function of a Riemann surface.}
Let $\mathcal R$ be an arbitrary algebraic Riemann surface defined as branched $s+1$-sheeted covering over $\overline { \CC}$ with canonic projection $ \pi: \ \mathcal R \to \overline \CC .$ Assume that
elements of the set $\pi^{-1} (\infty) = \{ \infty^{(i)}, i = 0,1, \dots, s\}$ are distinct. We fix one of them and denote $ \infty^{(0)}$.

Then there exist a unique function $g(\zeta): \mathcal R \to \overline \R$ with the following properties. The function $g$ is harmonic on the finite part of $\mathcal R$, as $\ z = \pi (\zeta) \to \infty$ we have
\begin{equation} \label{G-2}
\begin{aligned}
& g(\zeta) = -s \log|z| + O(1), \quad \zeta \to \infty^{(0)},  \\
&g(\zeta) = \log|z| + O(1), \quad \zeta \to \infty^{(i)} \quad i = 1, \dots, s 
\end{aligned}
\end{equation}
and, finally,  $g$ is normalized by $\sum g(z^{(i)})=0$.

We call $g$ the (real) $g$-function of $\mathcal R$ (with one marked point $\infty^{(0)}$).

Together with the real $g$ function we define a complex one $G = g + i \tilde g$ so that we have $g = \Re G.$ Function $G(z)$ is multivalued analytic function on $\mc R$; equivalently $G$ is a third kind Abelian integral with poles at infinities and divisor indicated in \eqref{G-2} above.

In many instances it is convenient to identify the coordinate $\zeta \in \mc R$ with its projection $z = \pi(\zeta)$. We will use such identification when it can not lead to ambiguities. Notation $z^{(k)}$ (for some elements of $\pi^{-1}(z)$ may be used if numeration sheets is defined or irrelevant.

The derivative $G'(z)$ of the complex $G$-function is meromorphic
(rational) function on $\mc R$. In other words $G'$ is an algebraic
function. Recall that in case $s = 1$ we had $G' = \sqrt{V/A}$, so that
$w = G'$ is solution of quadratic equation $ Aw^2 - V = 0$ where $A$ is
polynomial with roots at branch points and $V$ is another polynomial
which may be determined.

Not much was known so far for the case $s >1$. There are a few isolated
results showing the situation is much more complicated. Next we discuss
the Angelesco case which is in many ways simpler then the general one.



\subsubsection{Riemann surface: existence theorem.}
In the Angelesco case Riemann surface $\mathcal R$  associated with
extremal vector - compact set $\vec\Gamma =\vec \Gamma (\vec f)$ may be
constructed for a given vector compact set $\vec\Gamma$ with $S$-property
using a procedure quite similar to the construction of hyperelliptic
Riemann surface $\mc R(f)$ described in section 4.2.5 above (In place
of the Green functions we have to use equilibrium potentials).  We
state the final result of this procedure as a theorem.

\begin{thm}\label{RS}
Let $\vec\Gamma =\vec \Gamma (\vec f)$ be the extremal vector compact set associated with vector function $\vec f \in \mc A$ and the Angelesco condition is satisfied for $\vec f$. Let $\lambda = \lambda_1 + \dots +\lambda_s$ where  $\vec\lambda_{\vec \Gamma} = (\lambda_1, \dots, \lambda_s)$ and
$\Gamma = \Gamma^1_1\cup\dots\cup\Gamma^1_s$.

Then the total equilibrium potential $U^\lambda(z)$ has a harmonic
continuation from the domain $\Omega = \ol {\CC} \setminus \Gamma$ to a
Riemann surface $\mathcal R$ which is a $s+1$ sheeted branched over the
sphere with canonic projection $\pi: \mc  \R \to \ol  \CC$. Moreover,
with
$$
\pi^{-1}(\Omega)=\Omega_0\cup\Omega_1\cup\dots\cup\Omega_s
$$
we have
\begin{equation}\label{G-3}
U^\lambda(\pi(\zeta)) = g(\zeta) + C, \quad \zeta \in \Omega_0;
\quad  w_k - U^{\lambda_k}(\pi(\zeta)) =  g(\zeta) + C, \quad \zeta \in \Omega_k
 \end{equation}
 where $k = 1,2 \dots, s$ and numeration of (disjoint) domains
$\Omega_k \subset \mc R$ is such that  projection of $\partial \Omega_k \subset \partial \Omega_0$ is $\Gamma^1_k = \supp \lambda_k$ ( $\Omega_0 $ and $\Omega_k $ are connected through $\Gamma^1_k$).
 \end{thm}
\begin{proof}
The construction of  $\mc R$ is based on the equilibrium conditions and
the $S$-properties of the extremal compact set (see \eqref{10} and
\eqref{11} and uses the same standard procedure of gluing of $\mc R$
from $s+1$ (closures of) plane domains $ \Omega_k = \CC \setminus
\Gamma_k^1,$ where $k= 1,2, \dots, s$ and the domain $\Omega_0 = \CC
\setminus \cup \Gamma_k^1.$  which has been used in sec 4.2.5 for the
case $s = 1.$

The closure $\overline \Omega_k$ is defined on the basis of the inner metric in $\Omega_k$ same way as in \eqref{Dist} and \eqref{Clo} where in place of $\Gamma_k$  we use $\Gamma_k^1$ for positive $k$ and $\Gamma$ for $k=0$.

Then, again like in case $s=1$, we introduce the equivalency relation $\sim$ in the topological sum $\Sigma = \coprod_{k=0}^s \overline \Omega_k$ joining in a standard way the copy of $(\Gamma_k^1)^+$ on the sheet zero with the copy of $(\Gamma_k^1)^-$ on the $k$-the sheet and vise versa. Finally, we define $\mathcal R = \Sigma / \sim$  as a quotient space with respect to this equivalency. Local coordinates and the structure of a branched covering over $\CC$ are introduced on $\mc R$ in a standard way.

To prove \eqref{G-3} we use equilibrium conditions \eqref{13} and $S$-property  \eqref{14} for $\vec \Gamma$
\begin{equation*} \label{Asy-1}
 W_k(z): =
U^{\lambda_k + \lambda}(z) = w_k,  \quad
 \frac {\partial W_k(z)}{\partial n_1} = \frac {\partial W_k(z)}{\partial n_2}, \quad
 z \in \Gamma_k^1, \quad k = 1, 2, \dots, s
\end{equation*}
(second equality holds in interior points of arcs in $ \Gamma_k^1$).
Equivalently, we have
\begin{equation*} \label{}
U^\lambda(z) = w_k - U^{\lambda_k}(z),  \quad
 \frac {\partial}{\partial n_{1, 2}} U^\lambda(z)  =
 \frac {\partial}{\partial n_{2, 1}} \left(w_k - U^{\lambda_k}(z) \right)
 z \in \Gamma_k^1, \quad k = 1, 2, \dots, s
\end{equation*}
for $z \in (\Gamma_k^1)^0, \quad k = 1, 2, \dots, s.$ It follows, first, that the function $U(z)$  defined by
\begin{equation*} \label{}
U(z) = U^\lambda(z), \quad z \in \Omega_0,  \qquad
U(z) = w_k - U^{\lambda_k}(z),  \quad z \in \Omega_k,
\end{equation*}
has continuous extension to the whole Riemann surface $\mc R$. Moreover
$U(z)$ is continuously differentiable on the open analytic arcs in $
\Gamma_k^1$. Hence, $U$ is harmonic in the finite part of $\mc R$.
Finally asymptotics of $U$ at infinities is same as for $g$ in
\eqref{G-2}. Therefore $g(z) = U(z) + C$ on $\mc \R$.
\end{proof}

In terms of $G$-function the main theorem \ref{thm1} of the paper may be stated as follows
\begin{thm} \label{thmAsy-1}
Under the assumptions of the theorem \ref{thm1} for $\ k= 1, \dots , s$  we have convergence as $n \to \infty$
\begin{equation} \label{C}
C_{n,k}(z) =  \frac1n  \frac {Q'_{n,k}(z)}{Q_{n,k}(z)} = \int \frac {d \lambda_{n,k}(t)}{z-t}\  \to \
C^{\lambda_k} (z) =  \int \frac {d \lambda_{k}(t)}{z-t} = G'(z),  \quad z \in \Omega_k
\end{equation}
in plane Lebesgue measure $m_2$ and in capacity on compact sets in $\Omega_k = \ol\CC \setminus\Gamma_k^1$.
\end{thm}
The support $\Gamma_k^1$ of the measure $\lambda_k$ is a finite union of analytic arcs. 
In the complement domains $\Omega_k = \CC\setminus \Gamma_k^1$ each
function $w = C^{\lambda_k}(z)$ is a branch of the derivative $G'$ of
the $G$-function for $\mathcal R$.

Next theorem is also related to theorem \ref{thm1}, but it is not an immediate corollary of this theorem.

\begin{thm} \label{thmAsy-2}
Under the assumptions of the theorem \ref{thm1} for any branch of the remainder $R_n$ in any simpliconnected domain $\mc D \subset \mathcal R$ we have convergence as $n \to \infty$
in plane Lebesgue measure $m_2$ and in capacity on compact sets in $\Omega_k = \ol\CC \setminus\Gamma_k^1$.
\begin{equation} \label{R}
 \frac 1n \frac{R'_n(\zeta)}{R_n(\zeta)}\ \overset{m_2}{\to}\  \ G'(\zeta)\quad \text{and}\quad
 \frac 1n \log|R_n(\zeta)|  \ \overset{\mcap}{\to}\ g(\zeta), \quad z \in \mc D
\end{equation}
\end{thm}
A proof of this theorem would require at least a generalization of the part (ii) of the theorem \ref{thm3-1} to the case of Hermite--Pad\'e polynomials. Actually it contains more then that since part (ii) of the theorem \ref{thm3-1} assert convergence only on the main sheet of the Riemann surface.
The complete proof will take significant additional efforts and we do not present it in this paper.

\subsubsection{Riemann surface: algebraic equation.}
As a corollary theorem \ref{thmAsy-1} we derive algebraic equation whose solution is the derivative $G'$ of complex $G$-function. We state a theorem for the case $s = 2$ since it contains all the important details and at the same time it is much less crouded.

 \begin{thm}\label{thmEqu}
Let $\vec f = (f_1, f_2) \in \mc A$ be a vector function and the Angelesco condition is satisfied for $\vec f$. Let $G(z) = G(z; \vec f)$ be the $G$-function for the associated Riemann surface $\mc R(\vec f)$.
For $ j = 1,2$ let $A_j(z) = z^{p_j} + \dots $ be polynomials with roots at the branch points of $f_j$ (so that $p_j = \#(e_j)$ and $ p = p_1 +p_2$.

Then there exist polynomials $E(z) =  z^{p -2} + \dots $ and $F(z) =  z^{p -3} + \dots $ such that the derivative $w = G'(z)$ of $G(z; \vec f)$ is defined by equation
  \begin{equation} \label{Alg}
 A(z) w(z)^3 - 3 E(z) w(z) +2 F(z) = 0
\end{equation}
Moreover, all zeros of the polynomial $F$ are zeros of $G'$ and this polynomial may be representation in the form $F(z) = V_1(z) V_2(z) B(z) $ where $\deg V_1 = p_1 -2, \ $ $\deg V_2 = p_2 -2, \ $ and $B(z) = z - b.$  Zeros of $V_1$ and $V_2$ are identified with zeros of the quadratic differential whose trajectories constitute the components of the extremal compact set $\vec \Gamma$.
 \end{thm}
In the particular case when $p_1 = p_2 = 2$ the equation \eqref{Alg} has been presented in \cite{ApKuVa07} and it plays a fundamental role in this paper.

\begin{proof}
Let $\vec \Gamma(\vec f) = (\Gamma_1, \Gamma_2)$ be the extremal vector-compact set for $\vec f$ and
$\vec \lambda = (\lambda_1, \lambda_2)$ be corresponding equilibrium measure.
By
$$
C^\mu = \int(x-z)^{-1} d\mu(x)
$$
we denote the Cauchy transform a measure $\mu$. Function $C^\mu$ is also the derivative of the complex potential of $\mu$. By theorem \ref{thmAsy-1} the three branches of the multivalued analytic function $G'(z)$  in $ \ol\CC \setminus  (e_1\cup e_2)$ have representations
\begin{equation} \label{C-1.1}
G'(z) = C^{\lambda_k} (z),  \quad z \in \Omega_k = \ol\CC \setminus \Gamma_k,
\quad \text{and}\quad G'(z) = C^{\lambda} (z),  \quad z \in \Omega = \ol\CC \setminus \Gamma
\end{equation}
where $\Gamma = \Gamma_1 \cup \Gamma_2$ and $\lambda = \lambda_1 + \lambda_2$.

Let $P(w, z) = w^3 + r_1(z) w(z) + r_0(z)$  be the polynomial in $w$ with rational coefficients $r_j(z)$ associated with algebraic function $G'$, that is $ P(G'(z), z ) = 0.$ It follows from \eqref{C-1.1} that
\begin{equation} \label{C-1}
P(w, z) = (w - C^{\lambda_1} (z)) (w - C^{\lambda_2} (z)) (w + C^{\lambda} (z))
\end{equation}
and, therefore,
\begin{equation} \label{C-2}
r_0(z) =  C^{\lambda_1} (z) C^{\lambda_2}(z) C^{\lambda} (z)
\quad \text{and}\quad r_1(z) =
- C^{\lambda_1} (z)  C^{\lambda_2} (z)+ ( C^{\lambda} (z) )^2
\end{equation}
Each function $C^{\lambda_k} (z)$ has singularity at any $z = a \in e_k$ of the form $c(z)(z-a)^{-1/2}$
where $c(z)$ is analytic and $\ne 0$ at $z = a$ and the same is true for $C^{\lambda} (z)$ and any
$a \in e_1 \cup e_2$. It follows that each function $r_j(z)$ has at most  a simple pole at any $a \in e_1 \cup e_2.$  On the other hand as $z \to \infty$ we have $z C^{\lambda_1} (z) \to -1$ for $k =1,2$ and
$z C^{\lambda} (z) \to -2$. From here $r_0(z) = 2/z^3 +O(1/z^4)$ and $r_1(z) = -3/z^2 +O(1/z^3)$
and therefore $r_0 = 2F/A$ and $r_1 = -2E/A$ where $F$ and $E$ are some monic polynomials of degrees $p-3$ and $p-2$ respectively. Thus, after multiplication by $A$ polynomial $P(w, z)$ may be reduced to polynomial in two variables $w, z$ of the form indicated in \eqref{Alg}.

It follows from \eqref{Alg} that projections of zeros of algebraic
function $w = G'$ coincide with zeros of the polynomials $F(z) = 2z
^{p-3} + \dots.$ Function $G'(\zeta)$ is rational function of order $p$
on $\zeta \in\mc R$ (assuming that there are no cancelations) and,
therefore, it should have $p$ zeros on $\mc R$. Out of them it has $p_k
-2$ zeros at zeros of polynomial $V_k$ associated with $\Gamma_k$ where
$k = 1,2$. It has a simple zero at each copy of $\infty$ and this
accounts for $p-1$ zeros. It remains one more zero whose projection on
the plane we have denoted by $b$.
\end{proof}

\subsection{ Subclass $\LL \subset \mc A.$ Laguerre type equation. Conjecture.}
There are several subclasses of class $\mc A$ which may be considered
as generalizations of classical weights (their Cauchy transforms).
Corresponding Pad\'e and Hermite--Pad\'e polynomials are often called
semiclassical, since certain properties of classical orthogonal
polynomials are preserved for them. An important example is the class
$\LL$ defined as follows. For fixed set $e = \{a_1,\dots,a_p\}$ of
$p\geq 2$ distinct points we denote by $\LL_{e}$ class of functions $f$
of the form
\begin{equation} \label{L}
\LL_{e}= \left\{f(z)=f(z;\alpha):=\prod_{j=1}^p(z-a_j)^{\alpha_j}:\quad
\alpha_j\in\CC\setminus\Z,\ \right\}
\end{equation}
We also assume that $\sum_{j=1}^p\alpha_j=0$ and branch at infinity is selected by normalization $f(\infty)=1$. Clear that $\LL_e\subset \mathcal A(\ol\CC \setminus e).$

Let  $\LL = \cup \LL_{e}$ be the union of classes $\LL_{e}$ for all sets $e$.  We have $\LL \subset \mathcal A.$ Functions $f \in \LL$ are called sometimes the generalized Jacobi functions.

An important property of this class that for a fixed $e$ it contain
function with arbitrary branch type, so that the family of extremal
cuts associate with functions $\ f \in \LL_e$ is the same as
corresponding family for functions $\ f \in \mc A_e$.

Another important fact is that corresponding  Hermite--Pad\'e polynomials satisfy a Laguerre-type differential equation with polynomial coefficients. More exactly, let $\vec f = (f_0, f_1, \dots, f_s), \, f_k\in \LL$ and $q_{k,n},\, k= 0,1, \dots, s$, be associated Hermite--Pad\'e polynomials.  Then for any $n\in \N$ each of the following $s+1$ functions
$$w =  q_{0,n},\,\, q_{1,n}f_1,\,\,  \dots,\,\,  q_{s,n}f_s$$
 satisfy a linear differential equation
 \begin{equation}\label{DE}
\Pi_{s+1}(z)w^{(s+1)}+ \Pi_{s}(z)w^{(s) }+ \dots + \Pi_1(z)w'+\Pi_0(z)w = 0
\end{equation}
where coefficients $\Pi_k(z) = \Pi_{k,n}(z)$ are polynomials depending on $n$ whose degrees are altogether bounded by a number depending only on numbers of branch points of component functions $f_k$.

The equation has been derived in \cite{MaRaSu16} where some corollaries for
asymptotics has been obtained for simple particular functions $f \in \LL$
with a few branch points. After the paper was published a general
conjecture were developed by the authors of \cite{MaRaSu16} which we
present here in a somewhat abbreviated form. The two assertions below, if
proved, would create a general basis for investigation of asymptotics of
Hermite--Pad\'e polynomials.

{\bf Conjecture}. Let $\vec f = (f_1, \dots, f_s)$ and $f_k \in \mc A$. Let $A_k(z)$ be the monic polynomial with zeros at branch points of $f_k$ and $A(z) = \prod_{k=1}^s A_k(z)$.

Then there exists a finite set $\widehat {\mc R} = \widehat {\mc R}(A)$ of algebraic Riemann surface $\mc R$ depending only on $A$ which satisfy the following two condition.

(1) Projections of branch points of Riemann surface $\mc R$ belongs to the set of zeros of the polynomial $AV$ where $V(z)$ is some polynomial of degree $\deg A - 2$.

(2) Projections of poles of the derivatives $G'$ of the complex  $G$ function for $\mc R$ belong to the set of zeros of $A$.

Moreover, for a given $\vec f = (f_1, \dots, f_s)$ there is a unique
Riemann surface  $\mc R \in \widehat {\mc R}$ such that for any (properly normalized) branch of the remainder $R_n$ associated with $\vec f$  the sequence of functions $\frac 1n \log |R_n(z)|$ converges in capacity to the $g$-function of $\mathcal R \in \widehat {\mc R}$ in any domain $\mc D \subset\mathcal R$ assuming that $R_n$ has a single-valued branch in $\mc D$.\\


In the Angelesco case the part of the conjecture is proved in this
paper. It follows from the results of this paper that in the Angelesco
case we can go further on and make a complete description of $\mathcal
R(\vec f).$ For the general case such a complete description is known
only for a number of particular cases and it contains a large number of
details which we will not discuss here; see \cite{ApKuVa07}.

Convergence of the sequence of functions $\frac 1n \log |R_n(z)|$ in capacity in Angelesco case is stated above as theorem \ref{thmAsy-2}. As mentioned above, the proof would require significant additional efforts, but it is entirely within the range of the methods used in this paper.

In case of general assumptions on the configuration of branch points the conjecture is not proved yet even for functions from $\LL$ in \eqref{L}. However, it seems that all the necessary tools are available if the methods of  \cite{MaRaSu16} are combined with method presented in  \cite{Rak16} and in the current paper. We make a few comments which outline main steps of method, which may be used to prove the conjecture for class $\LL$.

First, we make the Riccati substitution $\quad u_n = \frac {w'_n}{n\
w_n}$ in the equation \eqref{DE} which reduces this equation to a
nonlinear differential equation with rational coefficients depending on
$n$. The obtained equation may be normalized in such a way that
families of coefficient functions become compact set with respect to
uniform convergence in spherical metric. This makes families of
solutions the compact set with respect to convergence in plane measure $m_2$.
In particular, any sequence of functions $R'_n/(n R_n)$ has
$m_2$-convergent subsequence and same is true for Cauchy transforms
$C_{n,k}$ of counting measures of Hermite--Pad\'e in \eqref{C}.

It is comparatively easy to prove that any limit equation is algebraic.
The original differential equation was obtained as equality to zero of
some Wronskian. Riccati substitution reduces this Wronskian to a
Vandermond determinant whose elements are convergent in $m_2$ measure.
This proves that the $m_2$-limit $C_k(z)$ of any convergent subsequence
of the sequences $C_{n,k}$ has the following property: there exist a
polynomial in two variable $P(w, z)$ (it does not depend on $k$) such
that $P(z, C_k(z)) = 0$ for $m_2$- almost all $z$ in plane.

There is a well known and partially proved conjecture (see
\cite{BeRu01}, \cite{BjBoBo11}) that equality $P(z, C^\mu(z)) = 0$ for
$m_2$-almost all $z$ in plane for a Cauchy transform of a positive
measure $\mu$  implies that the this measure is supported on finite
number of analytic arcs. If proved this would allow to assert that any
convergent subsequence of the sequences $R'_n/(n R_n)$ converges in
plane measure to an algebraic function $F$ which is a meromorphic
function on an algebraic Riemann surface $\mathcal R$ and that $F = G'$
where $G$ is the complex $G$-function for $\mathcal R$.

To complete the proof of Conjecture for class $\LL$ we need to prove that $\mathcal R$ does not depend on subsequence. This will be a corollary of the uniqueness of the Riemann surface $\mathcal R$ introduced in the conjecture. This part of the problem may be more difficult since it is related to a nontrivial combinatorics which we will not discuss here.



\def\cprime{$'$}
\providecommand{\bysame}{\leavevmode\hbox to3em{\hrulefill}\thinspace}
\providecommand{\MR}{\relax\ifhmode\unskip\space\fi MR }
\providecommand{\MRhref}[2]{%
  \href{http://www.ams.org/mathscinet-getitem?mr=#1}{#2}
}
\providecommand{\href}[2]{#2}

\end{document}